\documentclass[aap]{imsart}

\RequirePackage{amsthm,amsmath,amsfonts,amssymb}
\RequirePackage[numbers]{natbib}
\RequirePackage[colorlinks,citecolor=blue,urlcolor=blue]{hyperref}
\RequirePackage{graphicx}

\startlocaldefs

\usepackage{srcltx}

\usepackage{mathrsfs}

\renewcommand{\d}{{\rm d}}

 \newcommand{\e}{{\rm e} }

 \newcommand{\eps}{\varepsilon}

 \newcommand{\R}{\mathbb{R}}
 \newcommand{\N}{\mathbb{N}}
 \newcommand{\Z}{\mathbb{Z}}

 \newcommand{\E}{\mathbb{E}}
 \renewcommand{\P}{\mathbb{P}}
 \def\1{{\mathchoice {1\mskip-4mu\mathrm l} 
{1\mskip-4mu\mathrm l}
{1\mskip-4.5mu\mathrm l} {1\mskip-5mu\mathrm l}}}

\newcommand{\ignore}[1]{}
\newcommand{\heap}[2]{\genfrac{}{}{0pt}{}{#1}{#2}}

\newcommand{\ssup}[1] {{\scriptscriptstyle{({#1}})}}

\newtheorem{theorem}{Theorem}[section]
\newtheorem{lemma}[theorem]{Lemma}
\theoremstyle{remark}

\newtheorem{Assume}[theorem]{Assumption} 

\newtheorem{prop}[theorem]{Proposition}

\newtheorem{remark}{{\slshape\sffamily Remark}}[]

\endlocaldefs

\begin{document}

\begin{frontmatter}
\title{Central limit theorem for Gibbs measures on path spaces including long range and singular interactions and homogenization of the stochastic heat equation}
\runtitle{Gibbs measures on path spaces including long range and singular interactions}

\begin{aug}
\author[A]{\fnms{} \snm{Chiranjib Mukherjee}\ead[label=e1]{chiranjib.mukherjee@uni-muenster.de}},
\address[A]{Department of Mathematics and Computer Science,
University of Muenster}

\end{aug}

\begin{abstract}
We consider a class of Gibbs measures defined with respect to increments $\{\omega(t)-\omega(s)\}_{s<t}$ of $d$-dimensional Wiener measure, with
the underlying Hamiltonian carrying interactions of the form $H(t-s,\omega(t)-\omega(s))$ that are invariant under uniform translations of paths. 
In such interactions we allow {\it{long-range}} dependence in the time variable (including power law decay up to $t\mapsto (1+t)^{-(2+\eps)}$ for $\eps>0$) 
and {\it{unbounded (singular)}} interactions (including singularities of the form
 $x\mapsto 1/|x|^p$ in $d\geq 3$ or $x\mapsto \delta_0(x)$ in $d=1$) attached to the space variables. 
These assumptions on the interaction seem to be sharp and cover quantum mechanical models
like the  Nelson model and the polaron problem with ultraviolet cut off (both carrying bounded spatial interactions with
power law decay in time) as well as  the Fr\"ohlich polaron with a short range interaction in time but carrying Coulomb singularity
 in space. In this set up, we develop a unified approach for proving a central limit theorem for the rescaled process of increments for any coupling parameter 
 and obtain an explicit expression for the limiting variance which is strictly positive.

As a further application, we study the solution of the multiplicative-noise stochastic heat equation in spatial dimensions $d\geq 3$. When the noise is mollified both in time and space, we show that the averages of the 
diffusively rescaled solutions converge pointwise to the solution of a diffusion equation whose coefficients are homogenized in this limit.

\end{abstract}



\end{frontmatter}

\section{Introduction and motivation}

\subsection{Gibbs measures on increments of path spaces} The theory of Gibbs measures for lattice spin models or continuous point processes
is a well established subject of statistical mechanics \cite{G88,F88}.
On the other hand, Gibbs measures on {\it{path spaces}}, which are strongly inspired by their quantum mechanical applications,
 have also been studied extensively over the last decade-- see the classical book by Spohn \cite{S04} for a general reference, and \cite{N64,B03,BHLMS02,BL03,BLS05} for functional-analytic approaches
and Betz- Spohn \cite{BS05} and Gubinelli \cite{G06} (see also \cite{GL09}) for probabilistic approaches. Before turning to the background, let us shortly describe 
the mathematical layout of these models, which is 
obtained by perturbing the law $\P$ of $d$-dimensional Brownian increments by the exponential of a finite-volume energy of the form 
 \begin{equation}\label{H-scr}
 \mathscr H_T= -\int_0^T\int_0^T \d s \, \d t \,\, H\big(t-s, \omega(t)-\omega(s)\big)
 \end{equation}
 leading to the (finite volume) Gibbs measure 
  \begin{equation}\label{Q-beta-T}
  \d \widehat{\mathbb Q}_{\alpha,T}:= \frac 1 {Z_{\alpha,T}} \e^{-\alpha \mathscr H_{T}} \d\P. 
 \end{equation}
Here $\alpha>0$ is the {\it{coupling parameter}} (or the {\it{inverse temperature}})
and $\P$ is the law of a $d$-dimensional Brownian 
increments which is defined on the $\sigma$-algebra generated by $(\omega(t)-\omega(s)\colon 0 \leq s < t \leq T))$ on the time interval $[0,T]$.
Also, $Z_{\alpha,T}= \E^{\P}[\exp\{-\alpha \mathscr H_T\}]$
is the normalizing constant or the {\it{partition function}}. The goal of the present article is to establish existence and uniqueness of the infinite-volume Gibbs measures 
$\widehat{\mathbb Q}_\alpha:=\lim_{T\to\infty}\widehat{\mathbb Q}_{\alpha,T}$, analyze the behavior of this limit and prove central limit theorems of the distribution of the rescaled increments under both $\widehat{\mathbb Q}_{\alpha,T}$ and $\widehat{\mathbb Q}_\alpha$, i.e.,  to prove (see Theorem \ref{thm2}) for any coupling parameter $\alpha>0$,
\begin{equation}\label{eq:CLT}
\begin{aligned}
\lim_{T\to\infty}\widehat{\mathbb Q}_{\alpha,T} \,\, \bigg[\frac{\omega(T)-\omega(0)}{\sqrt T} \in \cdot \bigg] &=\lim_{T\to\infty} \widehat{\mathbb Q}_{\alpha} \bigg[\frac{\omega(T)-\omega(0)}{\sqrt T}\in \cdot\bigg]\\
&= \mathbf N_d(0, \sigma^2(\alpha) \mathbf I_{d\times d}) \quad\mbox{with}\quad \sigma^2 (\alpha) >0. 
\end{aligned}
\end{equation}
Our results hold for interactions which are rotationally symmetric in 
$x\mapsto H(\cdot,x)$, carry {\it long-range} dependence in the time variable $t\mapsto H(t,\cdot)$ (allowing a power law decay up to $(1+|t|)^{-(2+\eps)}$, see Assumption \ref{assumeA}), or include singularities in the space variable $x\mapsto H(\cdot,x)$ of Coulomb or Dirac type (see Assumption \ref{assumeB}). Using different methods, there have also been previous results (as already mentioned, see e.g. \cite{BS05,G06}) for Gibbs measures with interactions of the form \eqref{H-scr} showing validity of (functional) CLTs. The approach developed presently requires less restrictive assumptions on the interaction $H(\cdot,\cdot)$, covers spatial singularities in the latter, needs slower decay in time (which is conceivably sharp), and does not require the coupling parameter $\alpha$ to be small for the validity of the CLT (cf. Section \ref{comparison} for  a discussion on previous results and the present method of proof). Finally, Theorem \ref{thm2} has been recently applied in \cite{DS19} for rigorously establishing the relation $m_{\mathrm{eff}}(\alpha)=1/\sigma^2(\alpha)$ between the so-called {\it effective mass} and the CLT variance see Remark \ref{rmk-mass}). Also, the existence of the infinite volume limit 
$\lim_{T\to\infty}\widehat{\mathbb Q}_{\alpha,T}$ in Theorem \ref{thm2} has been used in \cite{MV18b} for showing convergence of rescaled  {\it strong coupling limit} 
$\lim_{\alpha\to\infty}(\lim_{T\to\infty}\widehat{\mathbb Q}_{\alpha,T})$ towards Pekar type process, see Remark \ref{sec-strong}. 

\subsection{Quantum mechanical background and the Nelson model} Before turning to formal statements of the main results, we would like to expand a bit on the physical background of the problems which inspired the present work. Gibbs measures of the  form \eqref{Q-beta-T} arise naturally while investigating the behavior of a quantum particle coupled to a scalar boson field, which was studied in a seminal work of Nelson \cite{N64} in the context of energy renormalization. Mathematically, the scalar boson field translates to an infinite-dimensional Ornstein-Uhlenbeck (OU) process $\{\varphi(x,t)\}_{x\in \R^d, t>0}$ with covariance structure 
\begin{equation}\label{eq0:Nelson}
\begin{aligned}
\int \varphi(x,t)\, \varphi(y,s) \, \mathbf P^{\mathrm{OU}}(\d\varphi)&=  \int_{\R^3} \d k\,\, |\widehat\rho(k)|^2\, \frac 1 {2\omega(k)} \, \e^{-\omega(k) |t-s|} \,\, \e^{\mathrm i k\cdot (x-y)} \\
&=: H(t-s, x-y).
\end{aligned}
\end{equation}
Here $\widehat\rho$ denotes the Fourier transform of the mass distribution of the quantum particle, while $\omega$ stands for the Phonon dispersion relation. 
 Now with a Hamiltonian $-\frac 12 \Delta+ e \varphi(x,t)$, the Feynman-Kac formula 
leads to the path measure  
\begin{equation}\label{eq1:Nelson}
\frac 1 {Z_T} \, \exp\bigg\{- e\int_0^T \varphi\big(W(t),t)\big) \d t\bigg\} \mathbf P^{\mathrm{OU}}(\d\varphi)\otimes\P(\d W) \qquad e>0.
\end{equation}
The exponent above is linear in $\varphi$. Thus integration w.r.t. the Gaussian measure $\mathbf P^{\mathrm{OU}}$, together with \eqref{eq0:Nelson} and \eqref{eq1:Nelson} 
now leads to the Gibbs measure $\widehat{\mathbb Q}_{\alpha,T}$ on the Wiener space defined in \eqref{Q-beta-T} for $\alpha=e^2/2$. 
The interaction \eqref{eq0:Nelson} above for the Nelson model encompasses a large class potentials $H(\cdot,\cdot)$ and to put our work into context, we will 
discuss two particular cases of physical prominence. First, when the dispersion relation $\omega$ in \eqref{eq0:Nelson} is chosen to be $\omega(k)=|k|$ and $\widehat \rho(\cdot)$ is radially symmetric, 
the Nelson model corresponds to the case of {\it massless Bosons} for which the interaction potential becomes 
$H(t,x)= \int_0^\infty \d r \widehat\rho(r) \mathrm e^{- r|t|} \, |x|^{-1}\sin(r|x|)$. Now the choice $\widehat\rho(r)=\mathrm e^{-r}$ leads to $H(t,x)=\frac 1 {|x|^2+ (1+|t|)^2} $ which, although being bounded in the space variable, carries the aforementioned 
long-range dependence of the time decay $t\mapsto H(t,\cdot)$.
Next, another particular case of the Nelson model is that of the {\it Fr\"ohlich Polaron} \cite{S86} which is an electron coupled to the optical modes of an ionic crystal. Mathematically, the Polaron is a particular case of \eqref{eq0:Nelson} when we choose $\omega(k)=\omega_0$ and $\widehat\rho(k)=|k|^{-1}$  that makes $H(t,x)= \frac{\e^{-\omega_0 |t|}}{|x|}$, which although decaying exponentially fast in time, now carries a Coulomb singularity in space variable $x\mapsto H(\cdot,x)$. 
In light of the above discussion, in the present context we propose a general method of analyzing Gibbs measures on path spaces that descend from covariance structures like \eqref{eq0:Nelson} including unbounded and slowly decaying potentials discussed above. 

\subsection{Homogenization of the stochastic heat equation}
 As an application, we will also prove a homogenization statement 
 of the stochastic heat equation in $d\geq 3$ with multiplicative white noise smoothened in {\it both space and time}. When the spatial dimension is one, a lot of progress has been made concerning a full understanding of this equation on a precise level (see \cite{BC95,H13,GP17,SS10,AKQ,ACQ11}, see also \cite{BC98,CSZ15} for the two-dimensional case). 
 The pertinent statement in the present context can be formally summarized as follows. Let $\dot B$ denote Gaussian space-time white noise and let $\psi: [0,\infty)\to \R_+$ and $\phi:\R^d\to \R_+$ be two fixed mollifiers (i.e. they are both positive, smooth, rotationally symmetric 
 functions with compact support with total integral one). Let
 $\dot B_1(t,x)= \int_{\R^d}\int_0^\infty \psi(t-s) \phi(x-y) \dot B(s,y)$ denote the mollified noise and let $u_1$ be the solution of  
 the multiplicative noise equation $\partial_t u_1=\frac 12 \Delta u_1 + \beta \, \dot B_1(t,x) u_1$. We set 
 \begin{equation}\label{hat-u}
 \hat u_\eps(t,x)=u_1(\eps^{-2}t,\eps^{-1}x).
 \end{equation}
To motivate the discussion, it is instructive to examine the case if the noise $\dot B$ is mollified only in space (and left white in time with $\psi=\delta_0$). Then by Feynman-Kac formula, time-reversal and 
 by subtracting a {\it deterministic quadratic variation}, $u_1$ can be made a {\it martingale} in $t$:
 \begin{equation}\label{mart}
 \E_x\bigg[\exp\bigg\{\beta \int_0^t  (\phi\star \dot B)(t-s,W_s) \d s - \frac {\beta^2}2 t (\phi\star \phi)(0)\bigg\}\bigg].
\end{equation}
If $\mathbf P$ denotes the law of the noise $\dot B$, then in this (white in time) case, by Brownian scaling, it is easy to see that $\mathbf E^{\mathbf P}[\hat u_\eps(t,x)]=1$. 
In this case, we refer to \cite{MSZ16} for statements pertaining to a phase transition (in $\beta$) of the limiting solution $\hat u_\eps$ as $\eps\to 0$. Moreover, when $\beta>0$ is small, we refer to \cite{MU17,GRZ17,CCM19a,CCM19b} for studies on the fluctuations of $\hat u_\eps$ (and that of the related KPZ solutions) and relations to Gaussian free field, and to \cite{BM17} for a quenched central limit theorem for the path measures of SHE. When $\beta>0$ is large, a contrasting scenario emerges 
as the quenched path measure enters a fully localized phase, see \cite{BM18}. 

However, the present scenario  of space-time mollification is very different where the correction term in the Feynman-Kac representation of the solution is not even deterministic (in contrast to \eqref{mart}) and  the averaged solution admits the representation $ \mathbf E^{\mathbf P}[\hat u_\eps(t,x)]=\E^\P[\exp\{-\beta \mathscr H_T(W)\}]$ with $\mathscr H_T$ of the form \eqref{H-scr} (in contrast and as remarked earlier, this average is constantly one when the noise is left white in time). Then our main result implies that, for any 
 $\beta>0$, $t>0$ and $x\in \R^d$, the rescaled solutions $\hat u_\eps(t,x)=u_1(\eps^{-2}t,\eps^{-1}x)$ homogenize on average as $\eps\to 0$: 
 \begin{equation}\label{intro:homogenize}
\e^{-(\frac t{\eps^2}\theta_0+ \theta_1)}\,\, \mathbf E^{\mathbf P} [\hat u_\eps(t,x)] \to \overline u(t,x) \qquad\mbox{with} \qquad \partial_t \overline u= \frac 12 \mathrm{div}\big(\mathrm{a_\beta} \nabla \overline u\big)
  \end{equation}
for constants $\theta_0, \theta_1$ and  $\mathrm a_\beta$ is the diffusion coefficient, see Theorem \ref{thm1} for a precise statement. It is worth noting that $\hat u_\eps(t,x)$ solves 
\begin{equation}\label{hatueps}
\partial_t \hat u_\eps= \frac 12 \Delta \hat u_\eps+ \beta \eps^{-2} \dot B_1\bigg(\frac t {\eps^2},  \frac x\eps\bigg) \hat u_\eps(t,x).
\end{equation}
On the other hand, the scaling properties of the noise implies that $\eps^{-2} \dot B_1(\eps^{-2} t, \eps^{-1} x) \sim c_0 \eps^{(d-2)/2} \dot B(t,x)$ for some constant $c_0$. However, the limiting behavior of $\hat u_\eps$ as $\eps\to 0$ 
does not correspond to a {\it{weak-noise}} problem which would produce (in the limit $\eps\to 0$) a solution of the heat equation $\partial_t v= \frac 12 \Delta v$ with unperturbed diffusion constant. The statement \eqref{intro:homogenize} underlines that, while the noise formally goes to zero, due to space-time correlations present in the mollification, the noise intensity does influence the diffusivity of the limiting PDE, as reflected by \eqref{intro:homogenize}. Let us now turn to the precise statements of the announced results.

 \section{Main results} 
 
\subsection{Gibbs measures on increments of interacting Wiener paths}

 Let $\Omega=\Omega_T$ denote the space of continuous functions $\omega:[0,T] \to \R^d$ which is equipped with the law
of the {\it Brownian increments} $\P=\P_T$, i.e, $\P$ is defined {\it only on the $\sigma$-algebra generated by the increments}\footnote{Alternatively,  the base measure $\P$ can also be defined as follows. Let $C([0,T]^2;\R^d)$ be the space of continuous functions from $[0,T]^2$ taking values in $\R^d$ and $D\subset C([0,\infty)^2;\R^d)$ is the subset such that $x\in D$ if and only if $x(t,t)=0$ for all $t\in [0,\infty)$ and $x(s,u)+x(u, t)= x(s, t)$. Then $\P$ is the Gaussian measure defined on $D$ such that if $\{X(s, t)\}_{s,t\geq 0}$ is the coordinate mapping process, then $\E^\P [ X(s, t) ]=0$ and $\E^\P[X(s, t)^2]=t-s$ for $t\geq s$ and $X(s, t)$ and $X(r, u)$ are independent under $\P$ if and only if $(s, t)$ and $(r, u)$ are disjoint intervals. This process can therefore be conceived as being driven by the {\it increments of a $d$-dimensional Brownian motion} 
 $(\omega(t))_{t\geq 0}$ by setting $X(s, t)= \omega(t)- \omega(s)$.}
 $$
\big\{\omega(t)-\omega(s)\colon 0\leq s < t \leq T\big\}.
$$
All throughout the rest of the article, $\E=\E^{\P}$ will denote expectation w.r.t. $\P$.

 We consider the finite-volume Gibbs measures 
of the form 
  \begin{equation}\label{eq:Q}
  \begin{aligned}
&\d\widehat{\mathbb Q}_T=   \d\widehat{\mathbb Q}_{\alpha,T}= \frac 1 {Z_{\alpha,T}} \e^{\alpha \mathscr H_{T}} \d\P,\qquad
 \mathscr H_T= \int_0^T\int_0^T \d s \, \d t \,\, H\big(t-s, \omega(t)-\omega(s)\big), \\
&\qquad\qquad Z_T=Z_{\alpha,T}= \E^\P\big[\e^{\alpha \mathscr H_T}\big].
\end{aligned}
    \end{equation}
    We will impose the following conditions on the interaction $H(\cdot,\cdot)$ appearing in the Hamiltonian $\mathscr H_T$. 
    Let $H:[0,\infty)\times \R^d \to \R$ be a function which is radially symmetric in $x$ and satisfies either of the following two conditions. 
    
  \begin{Assume} \label{assumeA}
  There exists $C\in (0,\infty)$ and $\eps>0$ such that 
  \begin{equation}\label{eq:assumeA}
  \sup_{x\in \R^d} |H(t,x)| \leq \frac {C}{ (1+t)^{2+\eps}}.
  \end{equation}
  \end{Assume}

{\it Alternatively}, we will impose the following condition on $H(\cdot,\cdot)$.

  \begin{Assume} \label{assumeB}
  There exist functions $\rho:[0,\infty)\to \R_+$  and $V:\R^d\to \R$ such that $H(t,x)= \rho(t) V(x)$ where $\rho$ is bounded with compact support and 
  \begin{equation}\label{eq:assumeB}  
  V(x)=
  \begin{cases}
\mbox{  bounded (i.e. } \sup_{x\in \R^d} |V(x)| <\infty) \mbox{ for any } d\in \N, \mbox { or }\\
\frac 1 {|x|^p} \mbox{ for any } p \in (0,\frac 2 {d-2}) \mbox{ and } d\geq 3, \mbox{ or }\\
\delta_0(x) \mbox{ for } d=1.
\end{cases}
\end{equation} 
\end{Assume}


 Recall the measure $\widehat{\mathbb Q}_T$ from \eqref{eq:Q}. Here is our first main result, for which we will need to impose   
 
\begin{theorem}\label{thm2}
Impose  \underline{either} Assumption \ref{assumeA} \underline{or} Assumption \ref{assumeB}. Then, for any $\alpha>0$, the limit 
$\widehat{\mathbb Q}_{\alpha}:=\lim_{T\to\infty}\widehat{\mathbb Q}_{\alpha,T}$ exists and weakly as probability measures, 
\begin{equation}\label{eq-thm2}
\begin{aligned}
\lim_{T\to\infty}\widehat{\mathbb Q}_{\alpha,T} \bigg[\frac{\omega(T)-\omega(0)}{\sqrt T}\in \cdot\bigg] &=\lim_{T\to\infty}\widehat{\mathbb Q}_{\alpha} \bigg[\frac{\omega(T)-\omega(0)}{\sqrt T}\in \cdot\bigg] \\
&= \mathbf N\big(\mathbf 0, \sigma^2(\alpha)\mathbf{I}_{d\times d}\big), \qquad \sigma^2(\alpha)>0.
\end{aligned}
\end{equation} 
Here $ \mathbf N\big(\mathbf 0, \sigma^2\mathbf{I}_{d\times d}\big)$ stands for a centered $d$-dimensional Gaussian law with covariance matrix $\sigma^2\mathbf{I}_{d\times d}$. \end{theorem}

\begin{remark}[\bf Effective mass of a quantum particle]\label{rmk-mass}
A physically relevant quantity is the so-called {\it effective mass} which is intimately related with the diffusion co-effecient of the central limit theorem stated in Theorem \ref{thm2}. The effective mass is defined as the 
as the inverse of the curvature:
 $$
 m_{\mathrm{eff}}(\alpha)= \bigg[\frac{\partial^2}{\partial P^2}E_\alpha(P)\big|_{P=0}\bigg]^{-1}
 $$
where $E_\alpha(P)$ is the bottom of the spectrum of the (fiber) Hamiltonian (see e.g. \cite{S86,DS19} for background material) 
of the quantum particle ($E_\alpha(\cdot)$ is rotationally symmetric and 
for Polaron-type interaction (see Section \ref{sec-Polarontype}) as well as for the Fr\"ohlich Polaron (see Section \ref{sec-Polaron} below) it is known that $E_\alpha(P)$ is analytic when $P\approx 0$). Assuming the validity of a CLT with variance $\sigma^2(\alpha)$, Spohn \cite{S86} heuristically argued the validity of the identity  
\begin{equation}\label{massvariance}
m_{\mathrm{eff}}(\alpha)^{-1}= \sigma^2(\alpha).
\end{equation}
Very recently, using the CLT provided by Theorem \ref{thm2}, this identity has been verified rigorously for the Polaron-type interactions appearing in Section \ref{sec-Polarontype} (see \cite[Theorem 3.2]{DS19}), while using the CLT proved in \cite{MV18}, the same identity \eqref{massvariance} has been shown to hold (see \cite[Theorem 3.1]{DS19}) for the Fr\"ohlich Polaron (see Section \ref{sec-Polaron} below for a discussion). 
\end{remark} 

\begin{remark}[\bf Strong coupling limit]\label{sec-strong}
Strong coupling limit corresponds to studying the infinite volume limit $\widehat{\mathbb Q}_\alpha$ as $\alpha\to\infty$. By Brownian scaling, 
this is equivalent to studying $\lim_{\lambda\to 0} \widehat{\mathbb Q}_\lambda= \lim_{\lambda\to 0} \lim_{T\to\infty}\widehat{\mathbb Q}_{\lambda,T}$
where $\widehat{\mathbb Q}_{\lambda,T}$ descends from {\it Kac-type} interaction energy of the form $H(t,x)=\lambda \e^{-\lambda |t|} V(x)$. For these interactions, 
the infinite volume limit $\widehat{\mathbb Q}_\lambda= \lim_{T\to\infty}\widehat{\mathbb Q}_{\lambda,T}$ follows from Theorem \ref{thm2} under Assumption \ref{assumeA}
if $V:\R^d\to \R$ (for any $d\in \N$) is a continuous function vanishing at infinity. 
For such functions $V$ (which is included in the class of models defined below in Section \ref{sec-Polarontype}), or when $V(x)=\frac 1{|x|}$ in $d=3$ (which is the Fr\"ohlich Polaron defined in section \ref{sec-Polaron} for which the existence of infinite volume limit has been shown in \cite{MV18}), 
the vanishing Kac (or the strong coupling) limit $\lambda\to 0$ of $\widehat{\mathbb Q}_\lambda$  has been shown further in \cite{MV18b} to coincide with the increments of the stationary Pekar process, verifying a conjecture of Spohn \cite{S86}.\footnote{The Pekar process is a stationary diffusion process with generator $\frac 12\Delta+(\nabla\phi/\phi)\cdot\nabla$, with $\phi$ being the maximizer of the Pekar variational formula $\sup_{\|\phi\|_2=1}[\int\int_{\R^3\times \R^3} \phi^2(x)\phi^2(y)V(x-y) \d x \d y - \frac 12 \|\nabla\phi\|_2^2]$. Convergence of the infinite volume limit towards the increments of Pekar process also requires the latter variational problem to admit a solution $\phi$ which is unique modulo spatial translation, which is known to hold for $V(x)=1/ {|x|}$ in $d=3$ (see \cite{L76}). As shown in \cite{MV14,KM15,BKM15}, the Pekar process itself emerges as the limiting object of the ``mean-field version" of the Fr\"ohlich polaron, where the path measure is given by $\widehat\P^{\mathrm{mf}}_T(\d\omega)=\frac 1 {Z_T^{\mathrm{mf}}}\exp[\frac 1T\int_0^T\int_0^T \frac{\d s \d t}{|\omega(s)-\omega(t)|}]\P(\d\omega)$.}
\end{remark}

\subsection{\bf Applications to quantum mechanical models}   

Let us now turn to some concrete quantum mechanical models of interest. First remark that Assumption \ref{assumeA} allows the interaction $H(t,x)$ to have long-range dependence (power law decay) in the time variable and requires it to be bounded in the space-variable. However,  {\it it does not} require $H(t,x)$ to have admit any product structure (in time and space coordinates). If $\gamma=2+\eps$, then 
  $1+ t^\gamma + |x|^2 \geq 1+ t^\gamma \geq C (1+t)^\gamma$ for some $C>0$. 
 Hence, Assumption \ref{assumeA}  allows the interaction of the form 
 $$
 H(t,x)= \frac 1 {1+ t^{\gamma}+ |x|^2} \qquad \gamma>2
 $$
  which, as discussed above, is relevant for the Nelson model. We also believe that the power law decay $(1+t)^{2+\eps}$ is sharp : We refer to \cite{OS99} where (Gaussian type) interactions carrying the Boltzmann weight $\exp[-\rho(t)|x|^2]$ are considered and it is shown that if $\rho(t)\simeq (1+t)^{\-\gamma}$ for $\gamma\in (1,2]$ then there are at least two infinite volume Gibbs measures. Since uniqueness of Gibbs measures is refuted, one does not expect a central limit theorem to hold for such slowly decaying potential. \footnote{Since any limiting Gibbs measure(s) belong to the extreme points of stationary processes, having (at least) two such measures denies any of its convex combination to be an extreme point, thereby disallowing the convex combination to possess any ergodic or mixing properties. In the absence of the latter condition, one does not expect validity of a central limit theorem, see e.g. \cite{OS99}.}

  \smallskip

Recall the interaction \eqref{eq0:Nelson}: 
\begin{equation}\label{Htx}
H(t,x)= \int_{\R^d} \d k \, \frac{\widehat\rho(k)^2}{\omega(k)} \, \e^{-\omega(k) |t|}\, \e^{\mathbf i \, k\cdot x} 
\end{equation}
and fix the spatial dimension $d=3$ and rotationally symmetric functions $\widehat\rho$ and $\omega$. 
The following concrete models are of particular interest.

\subsubsection{\bf Nelson model}Using rotational symmetry, the interaction \eqref{Htx} simplifies to :\footnote{Using spherical symmetry and polar coordinates in $d=3$ we have 
$$
\begin{aligned}
H(t,x)&= \int_0^\infty\d r \int_0^{2\pi} \d\phi\int_0^\pi \d\theta r^2 \sin\theta \, \frac{\widehat\rho(r)^2}{\omega(r)} \e^{-\omega(r) |t|}\, \e^{\mathbf i r|x| \cos\theta}\\
&=2\pi\int_0^\infty \d r  \, \frac{r^2 |\widehat\rho(r)|^2}{\omega(r)} \frac {\e^{-\omega(r)|t|}}{\mathbf i r |x|}\,[\e^{\mathbf i r |x|}- \e^{\mathbf i r |x|}] 
= 4\pi \int_0^\infty \d r \frac{r |\widehat\rho(r)|^2}{\omega(r)} \, \e^{-\omega(r)|t|} \bigg(\frac{\sin(r|x|)}{|x|}\bigg).
\end{aligned}
$$}
$$
H(t,x)=4\pi \int_0^\infty \d r \frac{r |\widehat\rho(r)|^2}{\omega(r)} \, \e^{-\omega(r)|t|} \bigg(\frac{\sin(r|x|)}{|x|}\bigg)
$$
Following standard customs, if we choose 
$$
\omega(k)=|k|, \qquad \widehat\rho\in L^2(\R^d)
$$
with the support of $\widehat\rho$ bounded away from zero, then it readily follows that $\sup_{x\in \R^d}| H(t,x)| \leq C \e^{-\delta(|t|+1)}$ for some $\delta>0$, which is clearly satisfied by Assumption \ref{assumeA}. The choice $\omega(k)=|k|$ is sometimes refereed to as the acoustic Polaron.
 \smallskip 

\subsubsection{\bf Polaron type interaction with ultraviolet cut off}\label{sec-Polarontype}  Another physically relevant model of interest is a polaron type interaction which corresponds to the requirement 

\begin{equation}\label{Polarontype}
\omega(k) \geq \omega_0>0 \quad\mbox{ for some } \omega_0>0, \qquad {and }\,\,\, \widehat\rho\in L^2(\R^d) 
\end{equation}
in \eqref{Htx}. The condition $\|\widehat\rho\|_2<\infty$ is referred to as {\it ultraviolet cut off}. Then  it follows that $\sup_{x\in \R^d}| H(t,x)| \leq C\|\widehat\rho\|_2^2 \,  \e^{-\omega_0 |t|}$ which again is satisfied by Assumption \ref{assumeA}.

\smallskip 

\subsubsection{\bf Fr\"ohlich Polaron}\label{sec-Polaron} As discussed earlier, Fr\"ohlich Polaron model is defined  by choosing
\begin{equation}\label{Polaron}
\omega(k) =1 \qquad {and }\,\,\, \widehat\rho(k)= \frac 1 {|k|} \notin L^2(\R^d)
\end{equation} 
which leads to $H(t,x)=  \frac{\e^{- |t|}}{|x|}$. This case is not covered by Assumption \ref{assumeB}: note that although the Coulomb singularity of $x\mapsto \e^{-|t|}/|x|=H(t,x)$ in $d=3$ is covered by Assumption \ref{assumeB}, the latter requires a compactly supported function in the time variable (which does not cover the exponential decay in  $t\mapsto \e^{-|t|}/|x|=H(t,x)$). To handle this case, a completely different approach was developed in \cite{MV18} (see also \cite{MV18-II}) 
and an explicit description of the Gibbs measure in the infinite volume limit $T\to\infty$ was obtained,
together with a central limit theorem for the distribution of rescaled increments (both under the finite and infinite volume limit, as in \eqref{eq-thm2})
with an explicit formula for the variance $\sigma^2(\alpha) \in (0,1)$. In this approach, the identity $|x|^{-1}=\sqrt{2/\pi} \int_0^\infty \d u\exp\{-u^2 |x|^2/2\}$ was used to also show upper bound on the the variance $\sigma^2(\alpha) < 1$, implying the lower bound $m_{\mathrm{eff}}(\alpha)>1$ on the effective mass (see \eqref{massvariance}). 
Combining the above Gaussian representation with our main result, such an upper bound on the variance can also be obtained for the particular Coulomb interactions appearing in Assumption \ref{assumeB}.

\subsubsection{\bf Polymer models} Finally, the Dirac interaction considered in Assumption \ref{assumeB} covers prototypes of polymers measures, where self-intersections are rewarded (or penalized) by the energy ${\int_0^T\int_0^T}_{|s-t|\leq\eps}\delta(\omega(s)-\omega(t)) \d s \d t$ (see \cite{V69,R83,R85}).

\subsection{\bf Homogenization of the stochastic heat equation}

We now turn to the application to stochastic heat equation (SHE) with multiplicative noise. 
In the present context, 
we fix a spatial dimension $d\geq 3$ and denote by $\dot B$  a Gaussian space-time white noise defined on a complete probability space $(\mathcal X,\mathcal F,\mathbf P)$. In other words, if $\mathcal S(\R_+\times \R^d)$ denotes the space of Schwartz functions (i.e. smooth functions with derivatives having rapid decay at infinity) on $\R_+\times \R^d$, then for 
$\varphi\in \mathcal S(\R_+\times \R^d)$, 
$\dot B(\varphi)$ is a Gaussian random variable 
with mean $0$ and covariance function 
$$
\mathbf E^{\mathbf P}[ \dot B(\varphi_1)\,\, \dot B(\varphi_2)]= \int_0^\infty \int_{\R^d} \varphi_1(t,x) \varphi_2(t,x) \d x \d t, \qquad \varphi_1,\varphi_2\in \mathcal S(\R_+\times \R^d).
$$
 Throughout the rest of the article,
$\mathbf E$ will denote expectation w.r.t. the law $\mathbf P$ of the white noise $\dot B$.

In order to define  the noise pointwise in space and time, we fix any two non-negative, even, smooth and compactly supported functions $\psi:[0,\infty)\to \R_+$ and $\phi:\R^d\to \R_+$ which are normalized to have total mass $\int_0^\infty \psi(t) \,\d t=\int_{\R^d}\phi(x)\,\d x =1$, and set
$$
\dot B(t,x)= \int_{\R^d}\int_0^\infty \psi(t-s) \phi(x-y) \dot B(s,y) \d s \d y
$$
 to be the space-time mollified Gaussian noise. Fix any continuous and bounded function $u_0$ as initial condition and let $u(t,x)$ solve the multiplicative noise equation 
$$
\partial_t u_1=\frac 12 \Delta u_1 + \beta \, \dot B_1(t,x) u_1, \qquad u(0,x)=u_0(x).
$$
Its diffusively rescaled counterpart is defined as
\begin{equation}\label{hat-u}
\hat u_\eps(t,x)\stackrel{\mathrm{\ssup{def}}}{=}u_1(\eps^{-2}t,\eps^{-1}x)\qquad \hat u_\eps(0,x)=u_0(x).
\end{equation}
 Here is our next main result. 
\begin{theorem}[Annealed CLT and homogenization for the stochastic heat equation]\label{thm1}
Fix $d\geq 3$, $\beta>0$, $t>0$. Then there exist constants $\theta_0,\theta_1$ (depending on $\psi,\phi, \beta,d$) such that for any $x\in \R^d$, $\beta>0$ and $t>0$ 
\begin{equation}\label{homogenization}
\begin{aligned}
 \lim_{\eps\to 0} \e^{-\frac t{\eps^2}{\theta_0}- \theta_1} \,\, \mathbf E[\hat u_\eps(t,x)]
&= \overline u(t,x) 
\end{aligned}
\end{equation}
where $\overline u$ solves the homogenized diffusion equation
\begin{equation}\label{homogenized}
\partial_t \overline u= \frac 12 \mathrm{div}\big(\mathrm{a_\beta} \nabla \overline u\big)
\end{equation}
with diffusion coefficient $\mathrm a_\beta=\sigma(\beta)\mathrm{I}_{d\times d}$. 
\end{theorem}
We  remark that following standard customs the mollifiers $\psi$ and $\phi$ were chosen to be smooth and compactly supported. In this particular 
case the proof of Theorem \ref{thm1} only partially use Theorem \ref{thm2} when the time correlation function $\rho=\psi\star\psi$  has compact support and the spatial interaction $V=\phi\star \phi$ is bounded. Leveraging the full strength of Theorem \ref{thm2}, we could have as well worked with a more general class of mollifiers $\psi$ and $\phi$.  

We also remark that while  \eqref{homogenization} provides homogenization of the averages of $\hat u_\eps(t,x)$, it holds pointwise in $x\in \R^d$ and for any $\beta>0$. When $\beta>0$ is allowed to be small and spatial averages of  $\hat u_\eps(t,x)$ (i.e. integrals $\int f(x) \hat u_\eps(t,x) \d x$ for any smooth test function $f$) are considered, \eqref{homogenization} also implies a similar convergence  in probability. In fact, since $\int_{\R^d} \d x \,  \hat u_{\eps,t}(x) f(x)= \int_{\R^d} \d x \, [\hat u_{\eps,t}(x)-\mathbf E(\hat u_{\eps,t})] f(x)+ \int \d x \, \mathbf E[\hat u_\eps(t,x)] f(x)$,  after rescaling, \eqref{homogenization} readily provides the convergence of the second term to the homogenized limit. On the other hand, the first term can be handled by an $L^2(\mathbf P)$ computation that leads to studying exponential moments of functionals of two Brownian paths in the transient dimensions $d\geq 3$ and the existence of exponential moments requires the disorder strength $\beta>0$ to be small. 
Soon after our work was completed, circulated and posted, \cite{GRZ17} followed a strategy similar to the proof of Theorem \ref{thm2} (for the spatial case when $\rho$ and $V$ are both compactly supported, smooth and bounded) and proved the aforementioned convergence in probability for the spatial averages and showed that $\eps^{1-\frac d 2} \int_{\R^d} \d x \, f(x) [\hat u_\eps(t,x)- \mathbf E(\hat u_\eps(t,x))]$ converges in distribution to (spatially averaged) solution of the additive-noise heat equation (see also \cite{MU17} for a similar result with a different approach).

\subsection{\bf Earlier approaches and the central idea of the present proof}\label{comparison} 
Let us start with a brief overview of (probabilistic) methods that have been employed earlier for studying Gibbs measures on path spaces. 

\noindent {\bf Earlier approaches: } As remarked earlier, Gibbs measures with weights of the form \eqref{Htx} have been considered in the earlier works \cite{BS05,G06}. Requiring  that 
\begin{equation}\label{requirement}
\begin{aligned}
&\int_{\R^d} \d k \, |\widehat\rho(k)|^2 [\omega(k)^{-1}+ \omega(k)^{-2}+ \omega(k)^{-3}] < \infty, 
\\
&\int_{\R^d} \d k \,  |k|^2 \, |\widehat\rho(k)|^2 [\omega(k)^{-2}+ \omega(k)^{-4}] < \infty,
\end{aligned}
\end{equation}
the validity of a functional central limit theorem was shown in \cite{BS05,G06}.\footnote{In terms of the interaction \eqref{H-scr},  the CLT was shown in \cite[Theorem 17]{G06} under the assumption 
\begin{equation}\label{requirement2}
\begin{aligned}
&\sup_{x\in\R^d}|\nabla_x\nabla_x\, H(t,x) |\leq \frac{C}{(1+t)^\alpha} \,\,\mbox{for some}\,\,\alpha>3,\\
&\qquad\qquad\qquad\mbox{\underline{and}, if one of the following conditions holds:}\\
&\sup_{x\in\R^d} |H(t,x)| \leq\frac{ C^\prime}{(1+t)^{3+\eps}} \quad\mbox{ or, small coupling parameter }\alpha>0 \,\,\mbox{ and } \alpha>4,
\end{aligned}
\end{equation} 
for some $C,C^\prime\in (0,\infty)$ and $\eps>0$. In \cite[p.1578 and Eq. (2.10)]{DS19} the relation (resp. equivalence) of the two requirements \eqref{requirement} and \eqref{requirement2} is discussed.} In \cite{BS05}, the authors relied on using an auxiliary Gaussian field which allows to linearize the Ornstein-Uhlenbeck type interaction enabling one to view the ambient process  as the projection of a Markov process on a larger state space, or a Brownian motion moving in a dynamic random environment.  
The functional central limit theorem then follows from invoking Kipnis-Varadhan theory. In \cite{G06} the classical Dobrushin method \cite{D68},\cite{D70} has been employed 
 by cutting the path into several pieces and considering each piece as a spin and increments of Brownian paths as the basic variables. Then techniques from one-dimensional spin systems could be applied  to extract mixing properties of the limiting Gibbs measure which implies the desired central limit theorem. 
 
 For treating the Fr\"ohlich polaron corresponding to $H(t,x)=\e^{-|t|}/{|x|}$ (which does not satisfy \eqref{requirement}) 
 a new approach was developed in \cite{MV18} by expanding the exponential weight in the interaction 
 in a power series and writing the path measure as an explicit  mixture of Gaussian measures, the mixture being taken over a tilted Poisson point process with intensity measure 
 $\alpha \e^{-(t-s)} \1_{-T\leq s < t\leq T} \d s \d t$ and taking values in the space of intervals $\{[s,t]\}_{-T\leq s < t \leq T}$. By developing renewal theory for this ``mixture representation", an explicit formula (again as a mixture of Gaussian measures) for the infinite volume Polaron measure was obtained and a resulting CLT was shown under both the finite and infinite volume Polaron measure with a variance $\sigma^2(\alpha)\in (0,1)$. 
 
 Very recently, the approach of \cite{MV18} was extended in \cite{BP21} where the result on the 
 CLT requires (apart from the assumptions needed for existence of the infinite volume limit)
an additional hypothesis about quasi-concavity of $x\mapsto H(\cdot,x)$ to apply Gaussian correlation inequalities. These assumptions are satisfied by the Fr\"ohlich Polaron 
defined by $H(t,x)=\e^{-|t|}/{|x|}$ in $d=3$ (unlike Theorem \ref{thm2}), but are more restrictive than Assumption \ref{assumeA} (Theorem \ref{thm2} under Assumption \ref{assumeA} does not need any concavity in the spatial component $x\mapsto H(\cdot,x)$). The method there follows the approach developed in \cite{MV18} for the Fr\"ohlich Polaron, but unlike \cite{MV18} or our present approach, the proof of \cite{BP21} additionally needs known facts from quantum mechanics about the existence of ground states of the Fr\"ohlich Polaron at zero total momentum, spectral gaps etc.     

\medskip

\noindent{\bf Main idea of the current method:} In the present context we do not impose assumptions of the form \eqref{requirement} and develop an alternative
 approach based on a ``Markovianization technique"  which can be briefly summarized as follows. Let us first consider the case when the time correlations $\rho$ has compact support (cf. Assumption \ref{assumeB}). Then we can split the time interval $[0,T]$ into $O(T)$ many subintervals $I_j$ of
constant length so that in the double integral in $\mathscr H_{T}$ only interactions between ``neighboring intervals" $I_j$ and $I_{j+1}$ survive, while 
the diagonal interactions (i.e., interactions coming from the same interval $I_j$) are absorbed in the product measure $\P$  corresponding to Brownian increments on disjoint intervals. 
Then we are led to the study of a ``tilted" Markov chain taking values on the space of increments. 
Since our interactions also carry spatial singularities (like Coulomb or Dirac type) an important technical step at this point involves deriving  fine  regularity properties (on an exponential scale) of singular functionals of Brownian paths (cf. Lemma \ref{lemma:Lambda}). Leveraging these regularity estimates allows us to invoke ``uniform estimates" for unbounded interactions and consequently, it is justified that the aforementioned transformed Markov chain satisfies spectral gap estimates forcing its
fast convergence to equilibrium. Now when the time correlation function $\rho(\cdot)$ decays slowly (cf Assumption \ref{assumeA}) we can split the interval $[0,T]$ into subintervals $I_{j,L}$ of length $L=L(T)\gg 1$ such that $T/L\to \infty$ and the underlying measure $\widehat{\mathbb Q}_{\alpha,T}$ is well-approximated by a similar object that only captures interactions between neighboring intervals $I_{j,L}$ and $I_{j+1,L}$ while the diagonal interactions are again absorbed in the product measure $\P$ over disjoint intervals. This approximation is provided by good relative entropy estimates. We can now apply the above spectral gap technique from the first step to derive convergence of the tilted Markov chain to equilibrium {\it{uniformly in $L$}}, proving Theorem \ref{thm2}.

We hope that the present approach outlined above is conceptually simple and it covers a broad class of interactions satisfying Assumption \ref{assumeA} or Assumption \ref{assumeB}.  As this approach is different from the techniques relying on \cite{D68,D70}, 
it is able to handle spatial singularities, requires slower (and conceivably, sharp) decay of correlation in the time interaction, and the resulting CLT is shown to be valid for any 
coupling parameter $\alpha>0$.  
While the present set up does not cover interactions that are both singular in space and long-range dependent in time, we hope that such cases might yield to an extension of the present method, see Remark \ref{rmk-Coulomb}.

\medskip

\noindent{\bf Organization of the article.} The rest of the article is organized as follows. In Section \ref{sec-annealed-clt} we  construct a tilted Markov chain starting from interactions $H(t,x)$ which in the $t$-variable vanish outside a compact set, while possibly carrying singularities in the space variable and study the spectral properties of this Markov chain. Building on the results of Section \ref{sec-annealed-clt} and assuming that the singular interactions in Assumption \ref{assumeB}
can be handled in a uniform manner, the proofs of Theorem \ref{thm2} (both under Assumption \ref{assumeA} or Assumption \ref{assumeB}) and Theorem \ref{thm1} are provided in Section \ref{sec-thm2-thm1}. 
Finally, Section \ref{sec:lemma:Lambda} is then devoted to studying regularity properties of exponential functionals of Brownian paths with singular potentials that allow us to handle the singular interactions covered by Assumption \ref{assumeB}
 in a uniform manner.

\section{A tilted Markov chain on Brownian increments and its spectral properties}\label{sec-annealed-clt}


For conceptual clarity, throughout  the rest of this section we will assume that the interaction $H(t,x)$ satisfies Assumption \ref{assumeB} (i.e., it is of the form $H(t,x)=\rho(t)V(x)$, where the time correlation $\rho$  has compact support, while $V$ satisfies the one of the assumptions listed in \eqref{eq:assumeB}). The proof of Theorem \ref{thm2} under  Assumption \ref{assumeA} or Assumption \ref{assumeB} will be carried out in Section \ref{sec-thm2-thm1}
building on the approach we develop in the present section. 

\subsection{Markovianization of the interaction and a Hilbert-Schmidt operator}\label{sec-Markovian}
Without loss of generality, we will assume that $\rho(t)=0$ for $|t|>\frac 1 2$ and $T\in \N$ is an integer. Then $[0,T]=\cup_{j=1}^T I_{j-1}$ with $I_{j-1}=[j-1,j]$ and we will write for any $\omega\in \Omega$, 
\begin{equation}\label{eq:xi}
\xi_j=\xi_j(\omega)=\big\{\omega(t)-\omega(s)\colon s<t,\, s,t\in I_j\big\}.
\end{equation}
Recall that $\P$ is the law of Brownian increments defined  on $\sigma$-algebra generated by Brownian increments $\{\omega(t)-\omega(s)\}$ over disjoint intervals, and 
$Z_T$ is the total mass of the exponential weight which defines the measure $\widehat{\mathbb Q}_T$ over the time interval $[0,T]$(recall \eqref{eq:Q}). 
We can now rewrite 
\begin{align}
\widehat{\mathbb Q}_T\bigg(\prod_{j=1}^T \d \xi_{j-1}\bigg)&= \frac 1 {Z_T} \exp\bigg\{\sum_{j=1}^{T-1}\bigg(\int_{I_{j-1}}\int_{I_{j-1}} \d t \d s \, H\big(t-s, \omega(t)-\omega(s)\big)\bigg) \nonumber
\\&\qquad\qquad\qquad+ 
\bigg(\int_{I_{j-1}}\int_{I_{j}} \d t \d s \, H\big(t-s, \omega(t)-\omega(s)\big)\bigg)\bigg\} \P\bigg(\prod_{j=1}^T \d \xi_{j-1}\bigg)\label{eq:Q1.5}\\
&=\frac 1 {\mathscr Z_T} \exp\bigg\{\sum_{j=1}^{T-1} k(\xi_{j-1},\xi_j)\bigg\}\prod_{j=1}^T \pi(\d\xi_{j-1})\quad\mbox{with} \quad\mathscr Z_T= \frac {Z_T}{Z_1^T}, \label{eq:Q2}
\end{align}
where ``diagonal part of the interaction" is defined by 
\begin{equation}\label{eq:pi}
\pi\big(\d\xi_{j-1}\big)=\frac 1 {Z_1} \exp\bigg[\alpha\int_{I_{j-1}}\int_{I_{j-1}}\,\,\d t\, \d s\,\,H(t-s, \omega(t)-\omega(s))\bigg]\,\P(\d\xi_{j-1})
\end{equation}
as the weighted probability measure defined on increments over any single sub-interval $I_j$ (with $Z_1$ being the normalizing constant) and 
\begin{equation}\label{eq:k}
\begin{aligned}
k(\xi_{j-1},\xi_j) & = 2\alpha\int_{I_{j-1}} \d s \int_{I_{j}}\d t \,H \big(t-s, \,(\omega(t)-\omega(j-1))+(\omega(j-1)-\omega(s))\,\big)\\
&=2\alpha\int_{I_{j-1}} \d s \int_{I_{j}}\d t \,H \big(t-s, \omega(t)-\omega(s)\,\big)  
\end{aligned}
\end{equation}
is the kernel representing the ``off-diagonal part" of the interaction. We remark  that because of rotational symmetry of $x\mapsto H(t,x)$, we have $k(\xi,\xi^\prime)=k(-\xi,-\xi^\prime)$. 

Let $\mathcal E$ denote the space of increments $\big(\omega(t)-\omega(s)\big)_{0\leq s<t \leq 1}$ over the intervals of unit length and  
$$
L^2(\pi)=\Big\{u:\mathcal E\to \R\colon\,\, \int_\mathcal E u(\xi)^2\,\pi(\d\xi)<\infty\Big\}.
$$
On this space, we can define the integral operator 
\begin{equation}\label{def-operator-L}
(\mathscr L u)(\xi)= \int_{\mathcal E} \e^{k(\xi,\xi^\prime)}\, u(\xi^\prime)\,\pi(\d\xi^\prime)
\end{equation}
We also note that the kernel $\exp[k(\cdot,\cdot)]$ is not symmetric and hence, $\mathscr L$ is not necessarily a self-adjoint operator. We will now deduce some useful properties of $\mathscr L$. 
\begin{prop}\label{lemma-HS}
Under Assumption \ref{assumeB}, $\mathscr L$ is a Hilbert-Schmidt operator\footnote{In fact $\mathscr L \mathscr L^{\star}$ and $\mathscr L^{\star} \mathscr L$ are positive definite and are trace class operators with
$\mathrm{Tr} (\mathscr L\mathscr L^\star)=\mathrm{Tr} (\mathscr L^\star \mathscr L)= \|\mathscr L\|_{\mathrm{HS}}^2$.}
with 
$$
\|\mathscr L\|^2_{\mathrm{HS}} ={\int\int}_{\mathcal E\times\mathcal E}\,\, \e^{2k(\xi,\xi^\prime)}\,\,\pi(\d\xi)\pi(\d\xi^\prime)<\infty. 
$$
\end{prop}

If $V$  is a bounded function in Assumption \ref{assumeB}, the above fact follows immediately. In order to handle the spatially singular interactions in Assumption \ref{assumeB}, 
the proof of Proposition \ref{lemma-HS} will be carried out using Lemma \ref{lemma:Lambda} stated below for which we will need to introduce some further notation. 
Let $\P_x$ be the law of a $d$-dimensional Brownian path $W=(W_s)_s$ starting at $x\in \R^d$, while $\E_x$ stands for the corresponding expectation. 
For any  measurable function $V:\R^d\to \R$, we set 
$$
\begin{aligned}
&\Lambda_t(x)=\frac 1 t \int_0^t \d s \, V(W_s-x)=\int_{\R^d} V(x-y) L_t(\d y)= (V\star L_t)(x), \quad\mbox{where } \\
& L_t(A)=\frac 1 t \int_0^t \d s \, \1\{W_s\in A\} \,\,\forall A\subset \R^d.
\end{aligned}
$$
Writing $\langle \mu, f\rangle=\int_{\R^d} f(x) \mu(\d x)$ for any measure $\mu$ and any function $f$,  we then have 
  \begin{equation}\label{L1Linfty}
\int_0^1\int_0^1 V(W_t-W_s) \d s \, \d t =\langle \Lambda_1,L_1\rangle \leq \|\Lambda_1\|_\infty. 
\end{equation}
For $V(x)=\delta_0(x)$, in $d=1$ the above estimate is understood in the sense $\int_0^1\int_0^1 V(W_t-W_s) \d s \, \d t \leq \sup_{x\in \R} \int_0^1 \delta_0(W_t-x) \d t$.

The following technical fact will imply Proposition \ref{lemma-HS} as well as Lemma \ref{lemma-supinf} stated below. 
\begin{lemma}\label{lemma:Lambda}
Let $V:\R^d\to \R$ be any function satisfying \eqref{eq:assumeB} in Assumption \ref{assumeB}. Then for any $C>0$, 
\begin{equation}\label{check1}
\sup_{x\in \R^d} \E_x\big[\e^{C\|\Lambda_1\|_\infty}\big]<\infty.
\end{equation}
\end{lemma}
The proof of Lemma \ref{lemma:Lambda} is technical. In order to not ebb the flow of arguments, its proof  is deferred to until Section \ref{sec:lemma:Lambda}. 
Assuming the above fact, let us first conclude 

\smallskip 

\noindent{\bf Proof of Proposition \ref{lemma-HS} (Assuming Lemma \ref{lemma:Lambda}).} 
First, note that
\begin{equation}\label{eq1:lemma-HS}
\begin{aligned}
&{\int\int}_{\mathcal E\times\mathcal E}\,\, \e^{2k(\xi,\xi^\prime)}\, \pi(\d\xi)\pi(\d\xi^\prime)
\\
&=\frac 1{Z_1^2} \,\,{\int\int}_{\mathcal E\times\mathcal E}\,\,\exp\bigg\{4\alpha\int_0^1\d s\int_1^{2} \d t H(t-s,\omega(t)-\omega(s))\bigg\} \\
&\qquad\qquad\qquad\exp\bigg\{\alpha\sum_{j=1}^2\int_{I_{j-1}}\d s\int_{I_{j-1}} \d t{H(t-s, \omega(t)-\omega(s))}\bigg\}\,\,\P(\d\xi_0)\,\P(\d\xi_1)
\end{aligned}
\end{equation}
Let $H(t,x)= \rho(t) V(x)$ as in Assumption \ref{assumeB}. If $V$ is bounded, then we can simply estimate 
$\sup_{t,x} |H(t,x)| \leq \|\rho\|_\infty \|V\|_\infty$ which immediately implies Proposition \ref{lemma-HS}. 
Now assume that $V(x)=1/|x|^p$ for $p\in (0,2/(d-2))$ in $d\geq 3$ or $V(x)=\delta_0(x)$ in $d=1$. Then using $H(t,x) \leq \|\rho\|_\infty V(x)$ and non-negativity of $V$, 
it follows from \eqref{eq1:lemma-HS} that 
$$
\begin{aligned}
&{\int\int}_{\mathcal E\times\mathcal E}\,\, \e^{2k(\xi,\xi^\prime)}\, \pi(\d\xi)\pi(\d\xi^\prime) \\
&\leq 
\frac 1{Z_1^2} \,\,\E\bigg[\exp\bigg\{4\alpha\|\rho\|_\infty\int_0^{2}\int_0^{2}{\d t\,\d s}{V(\omega_{t}-\omega_s)}\bigg\}\bigg] <\infty.
\end{aligned}
$$
The finiteness of the last term above now follows from the estimate \eqref{L1Linfty} and Lemma \ref{lemma:Lambda}.\qed  

\smallskip

We will also have several occasions to use the following result.

\begin{lemma}\label{lemma-supinf}
Under Assumption \ref{assumeB}, there exist constants $C_1, C_2\in(0,\infty)$, such that 
\begin{equation}\label{eq-supinf}
\sup_{\xi\in\mathcal E}\int_{\mathcal E} \, \e^{k(\xi,\xi^\prime)} \, \, \pi(\d \xi^\prime) \leq C_1 \qquad\mbox{and } \,\,\,\,   {\inf_{\xi\in\mathcal E}\int_{\mathcal E} \, \e^{k(\xi,\xi^\prime)} \, \, \pi(\d \xi^\prime)}\geq C_2.
\end{equation}
\end{lemma}
\begin{proof}
Let us first prove the upper bound. Recall the definition of the probability measure $\pi$ and that of the kernel $k(\cdot,\cdot)$ from \eqref{eq:pi} and \eqref{eq:k}, respectively.  It follows that the left hand side is bounded above by 
$$
\sup_{\xi}\, \frac 1 {Z_1} \E\bigg[\exp\bigg(2\alpha \int_0^1\int_1^2 V(\omega(s)-\omega(t)) \d s \d t + \alpha \int_1^2\int_1^2 \d s \d t V(\omega(s)-\omega(t))\bigg)\bigg],
$$
with the supremum being taken over all $\xi=\xi(\omega)= \{\omega(t)-\omega(s): s,t\in[0,1],s<t\}$. Now, the estimate \eqref{L1Linfty} and Lemma \ref{lemma:Lambda} again dictate that the display above is bounded above by a constant $C_1\in (0,\infty)$, proving the desired upper bound. To prove the required lower bound, recall that under Assumption \ref{assumeB}, we have $H(t,x)=\rho(t) V(x)$ where $V$ is either bounded or it is non-negative, while $\rho$ is non-negative. It follows that there exists a constant $C^\prime \in [0,\infty)$ such that $H(t,x) \geq - C$, and hence for some constant $C_2\in (0,\infty)$ 
\begin{equation}\label{lbkernel}
\inf_{\xi,\xi^\prime} \e^{k(\xi,\xi^\prime)} \geq C_2 
\end{equation}
which provides the desired uniform lower bound.

\end{proof}

\subsection{The tilted Markov chain, spectral properties and the Krein-Rutman argument}

Lemma \ref{lemma-HS} and Lemma \ref{lemma-supinf} will now imply the following important ingredient regarding the operator $\mathscr L$.
\begin{lemma}\label{lemma-Perron-Frobenius}
There exists $\lambda_0>0$  and a unique strictly positive function $\psi_0: \mathcal E\to\R$ such that 
\begin{itemize}
\item \begin{equation}\label{eq-eigenfunction}
(\mathscr L\psi_0)(\xi)=\int_{\mathcal E} \,\e^{k(\xi,\xi^\prime)}\,\psi_0(\xi^\prime)\,\pi(\d\xi^\prime)=\lambda_0\psi_0(\xi), \qquad \int_{\mathcal E} \psi_0(\xi) \pi(\d\xi)=1.
\end{equation}
\item Furthermore, 
\begin{equation}\label{bound-eigenfunction}
\sup_\xi\psi_0(\xi)<\infty, \qquad\mbox{and }\quad \inf_\xi \psi_0(\xi)>0.
\end{equation}
\end{itemize}
\end{lemma}
\begin{proof}
 Lemma \ref{lemma-HS} implies in particular that $\mathscr L$ is a compact operator in $L^2(\pi)$. Furthermore, the uniform lower bound \eqref{lbkernel} 
implies that for any $u\in L^2_+(\pi)$, $\int \e^{k(\xi,\xi^\prime)} u(\xi^\prime) \pi(\d\xi^\prime) \geq C_2 \int u(\xi^\prime) \pi(\d\xi^\prime)$ and 
 $\mathscr L$ also maps the cone $L^2_+(\pi)$ of positive functions
into itself. Therefore, by Krein-Rutman theorem \cite{K74,JJ14}, 
there exists a simple eigenvalue $\lambda_0>0$ and an associated eigenfunction $\psi_0\in L^2_+(\pi)$ of $\mathscr L$ so that \eqref{eq-eigenfunction} holds and 
moreover, $\psi_0(\xi)\ne 0$ for
$\pi$-almost every $\xi$. Combining the last three assertions we have 
\begin{equation}\label{eq:positive}
\inf_\xi \psi_0(\xi)>0.
\end{equation}
On the other hand, the uniform upper bound part in Lemma \ref{lemma-supinf} also implies that the eigenfunction $\psi_0$ is bounded above. Indeed, by \eqref{eq-eigenfunction} and Cauchy-Schwarz inequality, for all $\omega\in\Omega$, 
\begin{equation}\label{ub-eigenfunction}
\begin{aligned}
\psi_0(\xi)= \lambda_0^{-1}\,\,\int_{\mathcal E} \, \e^{k(\xi,\xi^\prime)} \, \psi_0(\xi^\prime)\, \pi(\d\xi^\prime) 
&\leq \lambda_0^{-1}\,\|\psi_0\|_{L^2(\pi)}\,\,\,\bigg[\int_{\mathcal E} \, \e^{2k(\xi,\xi^\prime)}  \pi(\d\xi^\prime)\bigg]^{1/2} \\
&\leq \lambda_0^{-1}\,\|\psi_0\|_{L^2(\pi)}\,\,\,\sup_{\xi\in\mathcal E}\bigg[\int_{\mathcal E} \, \e^{2k(\xi,\xi^\prime)}  \pi(\d\xi^\prime)\bigg]^{1/2} \\
&\leq C_1\lambda_0^{-1}\,\|\psi_0\|_{L^2(\pi)}.
\end{aligned}
\end{equation}
\end{proof}
We will need another elementary fact. 
\begin{lemma}\label{lemma-submultiplicative} 
Let $0\leq f(\cdot)\leq 1$ be a measurable function on a measurable space $(X,\mathcal F)$ such that $|f(\theta)- f(\theta^\prime)| \leq C<\infty$. Let $\mu_1$ and $\mu_2$ be two probability measures on $X$ such that the signed measure $\eta:= \mu_1- \mu_2$ satisfies $\sup_{A\subset X} |\eta(A)| \leq \delta$ for some $\delta>0$. Then $|\int_{X} f \d\eta| \leq C\delta$. 
\end{lemma}
\begin{proof}
The assumption on $f$ implies that $\inf_a \sup_\theta |f(\theta)- a| \leq \sup_\theta | f(\theta)- \frac 12(\sup f+ \inf f)| \leq \frac C2$. On the other hand, let $X= P\cup N$ be the Hahn-Jordan decomposition of the signed measure $\eta= \eta^+- \eta^-$ with $\eta^+(N)=\eta^-(P)=0$. Since $\mu_1$ and $\mu_2$ are two probability measures on $X$, 
it follows that $\eta(X)= \mu_1(X)- \mu_2(X)=0$. Since the decomposition property implies $\eta^+(N)=\eta^-(P)=0$, it follows that $\eta^+(P)=\eta^-(N)$. Hence,
 $$
 |\eta|(X):= \eta^+(X)+ \eta^-(X) = \eta^+(P)+\eta^-(N)=2 \eta^+(P)= 2 |\eta(P)| \leq 2 \sup_{A\subset X} |\eta(A)| \leq 2\delta.
 $$
  Finally, since for any constant $a\in \R$, $\int_{X} a \d\eta=0$, we have 
$$
\bigg|\int_{X} f \d\eta\bigg| =\inf_a \bigg | \int_{X} (f-a) \d\eta\bigg | \leq \big(\inf_a \sup_\theta |f(\theta)- a|\big) \big(|\eta|(X)\big) \leq \big(\frac C2\big)(2\delta)= C\delta
$$
\end{proof} 

\begin{lemma}\label{lemma-tilted-MC}
Given the transition probability kernel $\widetilde \pi$ (defined below in \eqref{eq:pi:tilde}) on $\mathcal E$, there exists a unique $\widetilde\pi$-invariant probability measure 
$\mu(\d\xi)=\phi(\xi)\pi(\d\xi)$ on $\mathcal E$ (i.e. $\mu(\cdot)= \int \widetilde\pi(\xi, \cdot) \widetilde \mu(\d\xi)$). Consequently, 
there exists a unique stationary ergodic Markov chain $\mathbb Q^{\ssup{\widetilde\pi,\mu}}$ with starting distribution $\mu$
and transition probabilities constructed from $\widetilde \pi$.
\end{lemma}
\begin{proof}
Let us to define 
\begin{equation}\label{eq:pi:tilde}
\widetilde\pi(\xi,\xi^\prime)= \frac{\e^{k(\xi,\xi^\prime)} \,\,\psi_0(\xi^\prime) }{\lambda_0 \, \psi_0(\xi)}.
\end{equation}
Recall  \eqref{eq-eigenfunction} which implies that  
$\int \widetilde \pi(\xi,\xi^\prime) \pi(\d\xi^\prime)=1$ so that $\widetilde\pi(\cdot,\cdot)$ is a transition probability kernel w.r.t. the reference probability measure $\pi$. 
We will write $\widetilde\pi(\xi, \d\xi^\prime)= \widetilde\pi(\xi,\xi^\prime)\pi(\d\xi^\prime)$ and its corresponding $n$-step transition probability kernel as
\begin{equation}\label{eq:pi:tilde:n}
\widetilde\pi^{\ssup n}(\xi,\cdot)= \int \widetilde \pi(\xi,\d\theta) \, \widetilde\pi^{\ssup{n-1}}(\theta,\cdot) 
\end{equation}
Now the uniform lower bound on the kernel from \eqref{lbkernel} as well as uniform lower and upper bounds on the eigenfunction $\psi_0$ shown in \eqref{bound-eigenfunction} provide a {\it uniform lower bound} on the {\it tilted kernel}:
\begin{equation}\label{lb1}
\inf_{\xi,\xi^\prime} \widetilde\pi(\xi,\xi^\prime)= \inf_{\xi,\xi^\prime} \,\bigg(\frac{ \e^{k(\xi,\xi^\prime)} \psi_0(\xi^\prime)}{\lambda_0 \psi_0(\xi)}\bigg) \geq \delta \qquad\mbox{where}\,\,\delta:=\bigg(\frac {C_2}{\lambda_0}\bigg)\bigg(\frac{ \inf_\xi \psi_0(\xi)}{\sup_\xi \psi_0(\xi)}\bigg) >0.
\end{equation} 
The above lower bound implies that the transition probabilities $\widetilde\pi$ w.r.t. the reference probability measure $\pi$ 
satisfy the classical {\it D\"oblin condition} (see the seminal work \cite{D38}). This property has several consequences. In particular, if we set 
$$
\alpha_n:= \sup_{\xi,\xi^\prime} \sup_{A\subset \mathcal E} \big| \widetilde \pi^{\ssup n}(\xi, A)- \widetilde \pi^{\ssup n}(\xi^\prime, A)\big| 
$$
then, by \eqref{lb1} 
\begin{equation}\label{bound-alpha1}
\alpha_1 \leq (1-\delta).
\end{equation}  

By Chapman-Kolmogorov equation, 
$$
\widetilde\pi^{\ssup {n+m}}(\xi, A)- \widetilde \pi^{\ssup{n+m}}(\xi^\prime,A)=\int \widetilde\pi^{\ssup n}(\theta,A)[ \widetilde\pi^{\ssup m}(\xi,\d\theta) - \widetilde\pi^{\ssup m}(\xi^\prime,\d\theta)].
$$
 Now we want to use 
Lemma \ref{lemma-submultiplicative}  by choosing $f(\theta)= \widetilde\pi^{\ssup n}(\theta, \cdot)$, 
$\mu_1(\d\theta)= \widetilde\pi^{\ssup m}(\xi, \d\theta)$, $\mu_2(\d\theta)= \widetilde\pi^{\ssup m}(\xi^\prime, \d\theta)$. It follows that $\alpha_n$ is sub-multiplicative, 
i.e., $\alpha_{n+m} \leq \alpha_n \alpha_m$.  Hence, $\alpha_n \equiv 1$ for all $n$, unless $\alpha_k =:a< 1$ for some $k\in \N$. In the latter case, 
the above sub-multiplicative property dictates that $\alpha_n \leq \alpha_k ^{\small{[\frac nk]}} \leq C (a^{1/k})^n$ for some constant $C>0$. 
Then \eqref{bound-alpha1} implies that $\alpha_n \leq (1-\delta)^n\to 0$ as $n\to\infty$. Again by Kolmogorov-Chapman equation,  $\big|\widetilde \pi^{\ssup n}(\xi, A)- \widetilde\pi^{\ssup{n+m}}(\xi, A)\big| =\big| \int \big[\pi^{\ssup n}(\xi, A)- \widetilde \pi^{\ssup n}(\xi^\prime,A)\big] \, \widetilde\pi^{\ssup m}(\xi, \d\xi^\prime)\big| \leq \alpha_n \to 0$ as $n\to\infty$. Hence, there is a probability measure 
$$
\mu(\d\xi):=\phi(\xi) \pi(\d\xi) \qquad\mbox{with }\,\, \int \phi(\xi) \pi(\d \xi)=1
$$
with $\phi(\cdot)>0$ such that  
\begin{equation}\label{ergodic}
\sup_\xi \, \big\|\widetilde\pi^{\ssup n}(\xi, \cdot)-\widetilde \mu(\cdot)\|_{\mathrm{TV}} \leq C (1-\delta)^n \to 0 \qquad\mbox{as }\,\, n\to\infty\,\, \footnote{ The above convergence \eqref{ergodic} also holds uniformly over any initial condition $\psi\d\pi$ for $\int \psi \d\pi=1$ and $\psi\geq 0$.}\end{equation} 
where $\|\cdot||_{\mathrm{TV}}$ is the total variation norm (cf. \eqref{Tvdef}). It is also easily verified that $\mu$ is an invariant probability measure for $\widetilde\pi$, i.e., $\mu(\cdot)= \int \widetilde\pi(\xi,\cdot)\mu(\d\xi)$. In fact, if $\nu$ is another invariant probability measure, then by invariance $\nu(\cdot)= \int \widetilde\pi^{\ssup n}(\xi,\cdot)\nu(\d\xi)$ for every $n$, and passing to the limit $n\to\infty$ leads to the identity, $\nu(\cdot)=\lim_{n\to \infty} \int \widetilde\pi^{\ssup n}(\xi, \cdot) \nu(\d\xi)= \mu(\cdot)$ proving that the invariant probability measure $\mu$ is indeed unique. Given the $\widetilde\pi$-invariant probability measure $\mu$, 
there exists a unique stationary ergodic Markov chain whose law is denoted by $\mathbb Q^{\ssup{\widetilde\pi,\mu}}$. 
\end{proof}

Recall that for any two probability measures $\mu_1, \mu_2$  on any measurable space $(X, \mathcal F)$, their total variation distance on $\mathcal F$ is defined by 
\begin{equation}\label{Tvdef}
\|\mu_1-\mu_2\|_{\mathrm{TV}, \mathcal F}= \sup_{A\in \mathcal F} \big| \mu_1(A)- \mu_2(A)\big| = \frac 12 \sup_{\|f\|_\infty \leq 1} \big| \int f \d\mu_1- \int f \d\mu_2\big|
\end{equation}
with the supremum being taken over $\mathcal F$-measurable functions $f$ bounded above by $1$. 

Then the total variation distance between the original measure $\widehat{\mathbb Q}_n$ defined in \eqref{eq:Q2} and the stationary Markov chain 
$\mathbb Q^{\ssup{\widetilde\pi,\mu}}$ constructed in Lemma \ref{lemma-tilted-MC} is estimated as follows. 

\begin{lemma}\label{lemma2-TV-esti}
For $n_0<n$, let $\widehat{\mathbb Q}_{n_0,n}$ be the restriction of 
$\widehat{\mathbb Q}_n$ to the interval $[n_0,n]$, i.e., $\widehat{\mathbb Q}_{n_0,n}$ is the tilted measure \eqref{eq:Q2} defined on the $\sigma$-algebra $\mathcal F_{n_0,n}$ generated by Brownian increments in the time interval $[n_0,n]$. Then, 
\begin{equation}\label{eq:TV}
\lim_{n_0\to\infty} \, \sup_{n> n_0} \, \big\| \mathbb Q^{\ssup{\widetilde\pi,\mu}} \, - \, \widehat{\mathbb Q}_{n_0,n} \big\|_{\mathrm{TV}, \mathcal F_{n_0,n}} =0.
\end{equation}
\end{lemma}
\begin{proof}


\noindent{\bf Step 1:} First we determine the asymptotic behavior of the renormalized partition function $\mathscr Z_n=Z_n/Z_1^n$ defined in \eqref{eq:Q2}, see \eqref{eq:step1} below. 
Recall from \eqref{eq:pi:tilde} that we write $\widetilde\pi(\xi,\d\xi^\prime)= \widetilde\pi(\xi,\xi^\prime) \pi(\d\xi^\prime)$ for the tilted transition probabilities. Then 
$$
\prod_{j=1}^{n-1} \widetilde \pi(\xi_{j-1},\d\xi_j)= \frac 1 {\lambda_0^{n-1}} \frac {\psi_0(\xi_{n-1})}{\psi_0(\xi_0)} \, \e^{\sum_{j=1}^{n-1} k(\xi_{j-1}, \xi_j)} \, \prod_{j=1}^{n-1} \pi(\d\xi_j)
$$
and consequently, by definition of $\widehat{\mathbb Q}_n$ in \eqref{eq:Q2}
$$
\widehat{\mathbb Q}_n(\d\xi_0\dots\d\xi_{n-1})= \frac {\lambda_0^{n-1}}{\mathscr Z_n} \,\,\frac {1}{\psi_0(\xi_{n-1})} \,\,\bigg(\psi_0(\xi_0)\pi(\d\xi_0)\, \prod_{j=1}^{n-1} \widetilde \pi(\xi_{j-1},\d\xi_j)\bigg),
$$
or equivalently, 
$$
\exp\bigg[\log \bigg(\frac{\mathscr Z_n}{\lambda_0^{n-1}}\bigg)\bigg]\, \widehat{\mathbb Q}_n(\d\xi_0\dots\d\xi_{n-1})= \frac {1}{\psi_0(\xi_{n-1})} \,\,\bigg(\psi_0(\xi_0)\pi(\d\xi_0)\, \prod_{j=1}^{n-1} \widetilde \pi(\xi_{j-1},\d\xi_j)\bigg).
$$
Recall that $\int \psi_0(\xi)\pi(\d\xi)=1$. Integrating now both sides in the above display and invoking \eqref{ergodic} yields 
\begin{equation}\label{eq:step1}
\log\bigg(\frac{\mathscr Z_n}{\lambda_0^{n-1}}\bigg)= \log \bigg(\int\frac{\mu(\d\xi)}{\psi_0(\xi)}\bigg)+ o(1) \qquad\mbox{as } \, n\to\infty.
\end{equation} 
Recall that $\frac 1 {\psi_0}$ is a bounded function.

\smallskip 

\noindent{\bf Step 2:} For the restriction $\widehat{\mathbb Q}_{n_0,n}$ we also have similarly 
\begin{equation}\label{eq:step2}
\begin{aligned}
&\widehat{\mathbb Q}_{n_0,n}(\d\xi_{n_0}\dots\d\xi_{n-1})
= \frac {1}{\mathscr Z_n} \,\, \bigg\{\e^{\sum_{j=n_0+1}^{n-1} k(\xi_{j-1}, \xi_j)} \, \prod_{j=n_0+1}^{n-1} \pi(\d\xi_j)\bigg\} \\
&\qquad\qquad\times \bigg\{\bigg(\int (\psi_0(\xi_0)\pi(\d\xi_0)\int\pi(\d\xi_1)\e^{k(\xi_0,\xi_1)}\dots\int\pi(\d\xi_{n_0-1})\e^{k(\xi_{n_0-2},\xi_{n_0-1})}\bigg)\\
&\qquad\qquad\qquad\times \bigg(\e^{k(\xi_{n_0-1},\xi_{n_0})}\pi(\d\xi_{n_0})\bigg)\bigg\} \\
&= \frac {\lambda_0^{n-1}}{\mathscr Z_n} \bigg\{\prod_{j=n_0}^{n-2}\widetilde\pi(\xi_j,\d\xi_{j+1}) \frac{\psi_0(\xi_{n_0})}{\psi_0(\xi_{n-1})}\bigg\}\\
&\qquad\qquad\times\bigg\{\bigg(\int (\psi_0(\xi_0)\pi(\d\xi_0)\int \widetilde\pi(\xi_0,\d\xi_1)\int\widetilde\pi(\xi_1,\d\xi_2)\dots\int \widetilde\pi(\xi_{n_0-2},\d\xi_{n_0-1})\bigg)\\
&\qquad\qquad\qquad\times \frac{\widetilde\pi(\xi_{n_0-1},\d\xi_{n_0})}{\psi_0(\xi_{n_0})}\bigg\}\\
&=\frac {\lambda_0^{n-1}}{\mathscr Z_n} \bigg\{\prod_{j=n_0}^{n-2}\widetilde\pi(\xi_j,\d\xi_{j+1}) \frac{1}{\psi_0(\xi_{n-1})}\bigg\}\\
&\qquad\qquad\times\bigg\{\bigg(\int (\psi_0(\xi_0)\pi(\d\xi_0)\int \widetilde\pi(\xi_0,\d\xi_1)\int\widetilde\pi(\xi_1,\d\xi_2)\dots\int \widetilde\pi(\xi_{n_0-2},\d\xi_{n_0-1})\bigg)\\
&\qquad\qquad\qquad\times \widetilde\pi(\xi_{n_0-1},\d\xi_{n_0})\bigg\}\\
&=\bigg\{\prod_{j=n_0}^{n-2}\widetilde\pi(\xi_j,\d\xi_{j+1}) \,\bigg(\frac{\lambda_0^n}{\psi_0(\xi_{n-1}) \mathscr Z_n}\bigg)\bigg\} \bigg\{\widetilde\pi^{\ssup {n_0}}(\psi_0\d\pi, \d\xi_{n_0})\bigg\}.
\end{aligned}
\end{equation} 

\smallskip 

\noindent{\bf Step 3:} We can now conclude the proof of \eqref{eq:TV}. Recall that $\mathbb Q^{\ssup{\widetilde\pi,\mu}}$ is the law of the stationary Markov chain starting with the $\widetilde\pi$-invariant distribution $\d\mu=\phi\d\pi$.  In the last display above in \eqref{eq:step2} we first pass to the limit $n\to\infty$ and invoke \eqref{eq:step1} and \eqref{ergodic}, followed by letting $n_0\to\infty$ and invoking \eqref{ergodic} once more to complete the proof of \eqref{eq:TV}.

\end{proof}


The following general result will be useful in the present context. 

\begin{lemma}\label{DoeblinPoisson2}
Let $(X,\mathcal F, \mu)$ be a probability space equipped with a measurable and one-to-one map $S: X \to X$ which preserves $\mu$. Let $\mathbf S$ be the canonical extension of $S$ to $L^2(\mu)$ 
defined by $(\mathbf S g)(x)= g(Sx)$ for all $g\in L^2(\mu)$. 
Furthermore, assume the following conditions:
\begin{itemize}
\item Let $p(\cdot, \cdot)$ is a transition probability kernel on $X$ w.r.t. $\mu$, while $\mu$ is $p$-invariant, and there is a non-negative function $\delta(\cdot)$ on $X$ such that 
$\delta(y)\mu(\d y)$ is also $S$-invariant and 
for any $x\in X$, we have 
 \begin{equation}\label{Doeblin0}
 p(x,y) \geq \delta(y) \geq 0 \qquad\mbox{and}\quad \int \delta(y) \mu(\d y)=: \delta>0.
 \end{equation} 
\item The transition operator $\mathbf T$ defined w.r.t. the kernel $p(\cdot, \cdot)$  (i.e., for any $g\in L^2(\mu)$, $(\mathbf Tg)(x)= \int p(x,y) g(y) \mu(\d y)$)  commutes with $\mathbf S$, That is, 
$$
\mathbf T \circ \mathbf S= \mathbf S\circ \mathbf T. 
$$ 
\item Finally, $f\in L^2(\mu)$ is an odd function w.r.t. $S$: $f(Sx)= - f(x)$ for any $x \in X$.
\end{itemize}
Then there exists a solution $u\in L^2(\mu)$ of the Poisson equation 
\begin{equation}\label{Poisson}
 (\mathbf I- \mathbf T)u= f, \qquad\mbox{satisfying }\quad \|u\|_{L^2(\mu)}^2 \leq \delta^{-1} \|f\|_{L^2(\mu)}^2.
  \end{equation} 
Moreover, we have a lower bound on the Dirichlet form 
\begin{equation}\label{lb-Dir}
\sigma^2(f):= \int\int \big[u(y)- (\mathbf Tu)(x)\big]^2 \, p(x,y) \mu(\d x) \mu(\d y) \geq \int u^2(y) \delta(y) \mu(\d y) >0. 
\end{equation}
\end{lemma}

\begin{proof} We will prove the lemma in three steps. 

\noindent {\bf Step 1:} Since $f$ is an odd function w.r.t. $S: X\to X$ and $\mu$ is $S$-invariant, we have $\int f \d\mu=0$. 
Moreover, we claim that $\int \mathbf T f \d\mu=0$. Indeed, by invariance of $\mu$ w.r.t. $S$,
it suffices to show that $\mathbf Tf$ is also an odd function w.r.t. $S$. Indeed, $(\mathbf Tf)(Sx)=\int p(Sx,y) f(y) \mu(\d y)= \mathbf S(\mathbf Tf)(x)$, but since $\mathbf S\circ \mathbf T= \mathbf T \circ \mathbf S$,
we have from the last identity that $(\mathbf Tf)(Sx)= \mathbf T(\mathbf Sf)(x)= \int p(x,y) f(Sy) \mu(\d y)= - \int p(x,y) f(y) \mu(\d y)= -(\mathbf Tf)(x)$, which shows that 
$\mathbf Tf$ is odd w.r.t. $S$, and therefore together with
$S$-invariance of $\mu$, it follows that $\int \mathbf Tf\d\mu=0$. Continuing recursively, we then also have $\int \mathbf T^n f\d\mu=0$ for every $n\geq 0$.

\noindent {\bf Step 2:} Next, we prove existence of the solution $u$ in \eqref{Poisson} and the upper bound there. Since $\delta(y)\mu(\d y)$ is also $S$-invariant, we have 
$\int f(y)\delta(y) \mu(\d y)=0$. Thus, 
$$
\begin{aligned}
(\mathbf Tf)(x)=\int f(y) p(x,y) \mu(\d y)
&=  \int f(y) [p(x,y)-\delta(y)] \mu(\d y) \\
&= \int f(y) \sqrt{p(x,y)-\delta(y)} \sqrt{p(x,y)-\delta(y)} \mu(\d y).
\end{aligned}
$$
Now using $\int p(x,y)\mu(\d y)=1$ and $\int \delta(y) \mu(\d y)=\delta>0$, we have from Cauchy-Schwarz inequality that 
$$
\begin{aligned}
|(\mathbf Tf)(x)|^2 &\leq \bigg(\int |f(y)|^2 [p(x,y)-\delta(y)] \mu(\d y)\bigg)\,\, \bigg(\int  [p(x,y)-\delta(y)] \mu(\d y)\bigg) \\
&= (1-\delta) \bigg(\int |f(y)|^2 [p(x,y)-\delta(y)] \mu(\d y)\bigg). \\
\end{aligned}
$$
Integration of the (l.h.s.) above w.r.t. $\mu(\d x)$, together with $\mathbf T$-invariance of $\mu$ (i.e. $\int \mathbf Tg \d\mu=\int g\d\mu$ for every $g\in L^2(\mu)$) now yields
$$
\|\mathbf Tf\|_{L^2(\mu)}^2 \leq  (1-\delta) \|f\|_{L^2(\mu)}^2,
$$
proving that $\mathbf T$ is a contraction on the space of functions in $L^2(\mu)$ that are odd w.r.t. $S$. Since the latter space is $\mathbf T$-invariant, 
it follows that $\|\mathbf T^n f\|_{L^2(\mu)}^2 \leq (1-\delta)^n \|f\|_{L^2(\mu)}^2$ for all $n\geq 0$ and therefore, the Neumann series 
$u=\sum_{n=0}^\infty \mathbf T^n f$ is convergent, proving that $(\mathbf I- \mathbf T) u=f$ has a solution $u\in L^2(\mu)$ with $\|u\|_{L^2(\mu)}^2 \leq \delta^{-1} \|f\|_{L^2(\mu)}^2$. 

\noindent{\bf Step 3:} Let us now prove the desired lower bound \eqref{lb-Dir}. As remarked earlier, since $\int \mathbf T^n f\d\mu=0$ for all $n\geq 0$ and $u= \sum_{n\geq 0} \mathbf T^n f$, we have $\int u\d\mu=0$, and by $\mathbf T$-invariance of $\mu$, also $\int \mathbf T u\d\mu=0$.   Then again using \eqref{Doeblin0} and the last remark, we have 
$$
\begin{aligned}
\sigma^2(f)&= \int\int \big[u(y)- (\mathbf Tu)(x)\big]^2 \, p(x,y) \mu(\d x) \mu(\d y) \\
&\geq \int u^2(y) \delta(y)\mu(\d y)- 2 \bigg(\int u(y) \delta(y) \mu(\d y)\bigg)\bigg(\int (\mathbf Tu)(x) \mu(\d x)\bigg) + \delta \int (\mathbf T u)^2(x) \mu(\d x) \\
&\geq \int u^2(y) \delta(y)\mu(\d y).
\end{aligned}
$$
Since $f$ is not identically equal to $0$, $u$ is not identically equal to zero either and therefore the last term above is strictly positive. 
\end{proof}

\begin{remark}
We have imposed the above asymmetry condition on $f$ to provide a spectral gap in $L^2(\mu)$. However, if $f$ is assumed to be bounded, and it is mean-zero w.r.t. the invariant measure $\mu$,
then the spectral gap (in $L^\infty$) follows easily from the geometric ergodicity as in \eqref{ergodic}, which is a consequence of \eqref{Doeblin0}.
\end{remark}

\begin{remark}\label{MartCLT}
  We end with a standard fact which will be useful for the proof of Theorem \ref{thm2} provided below. 
  Let $(X_n)_{n\geq 0}$ be a Markov chain starting with an $p$-invariant probability distribution $\mu$ with transition kernel $p(\cdot,\cdot)$ 
  For any mean-zero function $f\in L^2(\mu)$, if there is a solution $u \in L^2(\mu)$ of the Poisson equation \eqref{Poisson}, then the rescaled additive functional 
  $\frac{S_n(f)}{\sqrt n}= \frac{f(X_1)+\dots + f(X_n)}{\sqrt n}$ can be approximated by $M_n(f)= \sum_{i=1}^n [u(X_i)- (\mathbf Tu)(X_i)]$ which defines a martingale (w.r.t. the canonical filtration $\sigma(X_i\colon 1\leq i \leq n)$) with stationary and ergodic $L^2(\mu)$ increments:
  $$
  \frac{S_n(f)}{\sqrt n}=\frac 1 {\sqrt n} \bigg[\sum_{i=1}^n u(X_i)- \sum_{i=1}^n (\mathbf T u)(X_i)\bigg]= \frac{M_n(f)}{\sqrt n} + \frac{(\mathbf T u)(X_0)- (\mathbf T u)(X_n)}{\sqrt n}.
  $$
  Since $u\in L^2(\mu)$, the correction term on the right hand side vanishes as $n\to\infty$ and by central limit theorem for martingale differences, $\frac{S_n(f)}{\sqrt n}$ converges to a centered Gaussian law. The ergodic theorem 
  implies that the variance is given by the Dirichlet form $\sigma^2(f)=\int\int [u(y)-\mathbf Tu(x)]^2 p(x,y) \mu(\d x)\mu(\d y)$ as in \eqref{lb-Dir}.\qed
  \end{remark}

\section{Proofs of Theorem \ref{thm2} and Theorem \ref{thm1}} \label{sec-thm2-thm1}

\subsection{\bf Proof of Theorem \ref{thm2} under Assumption \ref{assumeB}}



Recall that the ambient space $\mathcal E=(\omega(t)-\omega(s))_{0\leq s <t \leq 1}$ carries a reference probability measure $\pi$ defined in \eqref{eq:pi}. 
Let $S: X \to X$ be the map $S\xi=-\xi$. First remark that by our assumption on rotational symmetry of $x\mapsto H(\cdot, x)$, 
$\pi$ is $S$-invariant, and for the same reason, the kernel $k(\xi, \xi^\prime)$ is also $S$-invariant in the sense 
\begin{equation}\label{Sinv}
k(\xi,\xi^\prime)= k(S\xi, S\xi^\prime) \qquad\forall \xi,\xi^\prime \in \mathcal E.
\end{equation}
Next we want to show that the tilted kernel $\widetilde\pi(\cdot,\cdot)$ defined in \eqref{eq:pi:tilde} is also $S$-invariant in the above sense. Indeed, with
 with the operator $(\mathscr L g)(\xi)= \int \e^{k(\xi,\xi^\prime)} g(\xi^\prime) \pi(\d\xi^\prime)$ as before, and with $\mathbf S$ defined as in Lemma \ref{DoeblinPoisson2}, 
we have 
\begin{equation}\label{commute}
((\mathbf S \circ \mathscr L)g)(\xi)= \int \e^{k(S\xi, \xi^\prime)} g(\xi^\prime) \pi(\d\xi^\prime)= \int \e^{k(\xi,\xi^\prime)} g(S\xi^\prime) \pi(\d\xi^\prime)= \mathscr L(\mathbf Sg)(\xi)= ((\mathscr L \circ \mathbf S)g)(\xi).
\end{equation} 
If $\psi_0$ is the eigenfunction of $\mathscr L$ for the leading eigenvalue $\lambda_0$ obtained in Lemma \ref{lemma-Perron-Frobenius}, the same argument as above implies that $\mathscr L(\mathbf S \psi_0)= \lambda_0 \mathbf S\psi_0$,
which, together with uniqueness of $\psi_0$ enforces  $\psi_0= \mathbf S \psi_0$, or $\psi_0(\xi)= \psi_0(S \xi)$ for all $\xi\in \mathcal E$. Combined with \eqref{Sinv}, we therefore have the $S$-invariance 
of the tilted kernel $\widetilde\pi(\cdot,\cdot)$: 
$$
\widetilde\pi(S\xi, S\xi^\prime)= \widetilde\pi(\xi,\xi^\prime).
$$ 

Next, recall that $\mu(\d\xi)=\phi(\xi) \pi(\d\xi)$ and by \eqref{ergodic} and from the invariance of $\widetilde\pi$ just established, the the probability measure $\d\mu=\phi \d\pi$ is also $S$-invariant. 
Moreover $\mu$ is also invariant under the transition kernel $\widetilde\pi$ and   
since $\int \widetilde\pi(\xi,\xi^\prime)\pi(\d\xi^\prime)=1$, 
$$
p(\xi,\xi^\prime):= \frac{\widetilde\pi(\xi,\xi^\prime)}{\phi(\xi^\prime)}
$$
is a transition probability kernel w.r.t. $\mu$ and $\mathbb Q^{\ssup{\mu,\widetilde\pi}}$ is the law of the Markov chain with transition probabilities 
$p(\xi,\xi^\prime)\mu(\d\xi^\prime)$ starting with the distribution $\mu$. Since $\widetilde\pi(\xi,\xi^\prime) \geq \delta >0$, we also have 
$$
p(\xi,\xi^\prime)= \frac{\widetilde\pi(\xi,\xi^\prime)}{\phi(\xi^\prime)} \geq \frac\delta{\phi(\xi^\prime)}=:\delta(\xi^\prime), \qquad \mbox{and}\,\, \int \delta(\xi^\prime) \mu(\d\xi^\prime) = \delta \int \pi(\d\xi^\prime)=\delta>0
$$
so that the condition \eqref{Doeblin0} is satisfied. For any $g\in L^2(\mu)$, if $(\mathbf Tg)(\xi)= \int p(\xi,\xi^\prime) g(\xi^\prime) \mu(\d\xi^\prime)= \int \widetilde\pi(\xi,\xi^\prime) g(\xi^\prime) \pi(\d\xi^\prime)$,
then by the same argument as in \eqref{commute} and $S$-invariance of $\widetilde\pi(\cdot,\cdot)$, we have the desired commutation relation 
$\mathbf T \circ \mathbf S= \mathbf S \circ \mathbf T$. Finally, set $f(\xi)= \omega(1)- \omega(0)$ so that $f$ is odd w.r.t. $S$ and also $f\in L^2(\mu)$.\footnote{ Note that obviously $f$ has all moments w.r.t. the base measure 
$\P$ and it can be verified again  using Lemma \ref{lemma:Lambda}, Lemma \ref{lemma-supinf} and \eqref{ergodic} that square integrability of $f$ propagates through to the invariant measure $\mu$.}

With our earlier notation $\xi_j= (\omega(t)-\omega(s))_{j-1\leq s < t \leq j}$, we have $f(\xi_j)= \omega(j)-\omega(j-1)$ 
 and  by  Lemma \ref{DoeblinPoisson2} as well as Remark \ref{MartCLT}, the law of  the additive functional $\omega(n)-\omega(0)=\sum_{j=1}^n f(\xi_j)$
under the stationary Markov chain $\mathbb Q^{\ssup{\mu,\widetilde\pi}}$ with the $\widetilde\pi$-invariant distribution $\mu$ satisfies a central limit theorem: 
 \begin{equation}\label{MarkovCLT}
 \mathbb Q^{\ssup{\mu,\widetilde\pi}}\bigg[\frac{f(\xi_1)+\dots+ f(\xi_n)}{\sqrt n}\in \cdot\bigg] \Rightarrow N\big(0,\sigma^2\mathbf I_{d\times d}\big).
 \end{equation}
Since the above choice of $f$ is not a constant function, $\sigma^2>0$. Now Lemma \ref{lemma2-TV-esti} completes the proof of Theorem \ref{thm2} under Assumption \ref{assumeB} (assuming Lemma \ref{lemma:Lambda} whose proof will be presented in Section \ref{sec:lemma:Lambda}). 
 

\subsection{\bf Proof of Theorem \ref{thm2} under Assumption \ref{assumeA}}\label{sec:proof:assumeA}

We now turn to the proof of Theorem \ref{thm2} under Assumption \ref{assumeA} which demands $t\mapsto H(t,x)$  to have a polynomial decay at infinity, while $x\mapsto H(\cdot,x)$ remains bounded.  The proof will be carried out in three steps. 

\noindent{\bf Step 1:} Let us choose a parameter $L=L(T)$ such that as $T\to\infty$, we have $L\to \infty$, $n:=T/L \to \infty$ as well as $T^2/ L^{2+\eps}\to 0$, recall \eqref{eq:assumeA}. As before, we can divide the interval $[0,T]$ into $n$ subintervals 
$I_{j-1,L}=[(j-1)L,jL]$ of length $L$. Let $\mathscr I_T$ denote the set of all pairs $(I_{j,L},I_{r,L})_{j\ne r}$ such that $|t-s| >L$ for all $s\in I_j$ and $t\in I_r$ (i.e., $I_j$ and $I_r$ are separated by at least one interval $I_p$ for some $p=1,\dots,n$). 
Then $\#\mathscr I_T=O(n^2)$. 

With this notation, we can rewrite the measure $\widehat{\mathbb Q}_T$ in \eqref{eq:Q} as 
\begin{equation}\label{eq:Q.5}
\begin{aligned}
&\widehat{\mathbb Q}_T\bigg(\prod_{j=1}^n \d \xi_{j-1,L}\bigg)
= \frac 1 {Z_{L,T}} \exp\bigg\{\alpha\sum_{j=1}^{n-1}\bigg(\int_{I_{j-1,L}}\int_{I_{j-1,L}} \d s \d t H(t-s, \omega(t)-\omega(s)) \\
&\qquad\qquad\qquad\qquad\qquad\qquad\qquad\qquad+  2\int_{I_{j-1,L}}\int_{I_{j,L}} \d s \d t H(t-s, \omega(t)-\omega(s))\bigg) \\
&\qquad\qquad\qquad+\alpha \sum_{(I_{j,L},I_{r,L})\in\mathscr I_T} \int\int_{I_{j,L}\times I_{r,L}}\,\d s\, \d t H(t-s, \omega(t)-\omega(s)) \bigg\} \,\P\bigg(\prod_{j=1}^n \d \xi_{j-1,L}\bigg) \end{aligned}
\end{equation}
with $\xi_{j,L}=\xi_{j,L}(\omega)=\big\{\omega(t)-\omega(s)\colon s<t,\, s,t\in I_{j,L}\big\}$ denoting the increments on $I_{j,L}$. 
 
 \medskip 
 
Next, exactly as in \eqref{eq:Q2}, we can Markovianize the above interaction and define the transformed measure which is oblivious to interactions between intervals in $\mathscr I_T$. In other words we set 
\begin{align}
&\widehat{\mathbb Q}_{L,T}\bigg(\prod_{j=1}^n \d \xi_{j-1,L}\bigg)
= \frac 1 {Z_{L,T}} \exp\bigg\{\alpha\sum_{j=1}^{n-1}\bigg(\int_{I_{j-1,L}}\int_{I_{j-1,L}} \d s \d t H(t-s, \omega(t)-\omega(s)) \nonumber\\
&\qquad\qquad\qquad+ 2\int_{I_{j-1,L}}\int_{I_{j,L}} \d s \d t H(t-s, \omega(t)-\omega(s))\bigg)\bigg\} \,\P\bigg(\prod_{j=1}^n \d \xi_{j-1,L}\bigg) \label{eq:Q3}\\
&=\frac 1 {\mathscr Z_{L,T}} \exp\bigg\{\sum_{j=1}^{n-1} k(\xi_{j-1,L},\xi_{j,L})\bigg\}\prod_{j=1}^n \pi_L(\d\xi_{j-1,L})\nonumber
\end{align}

where
\begin{equation}\label{Q4}
\begin{aligned}
&Z_{L,T}= \E^\P\bigg[\exp\bigg\{\alpha\sum_{j=1}^{n-1}\bigg(\int_{I_{j-1,L}}\int_{I_{j-1,L}} \d s \d t H(t-s, \omega(t)-\omega(s)) \\
&\qquad\qquad\qquad\qquad\qquad\qquad+  2\int_{I_{j-1,L}}\int_{I_{j,L}} \d s \d t H(t-s, \omega(t)-\omega(s))\bigg)\bigg\}\bigg],\\
&\pi_L\big(\d\xi_{j-1,L}\big)=\frac 1 {Z_L} \exp\bigg[\alpha\int_{I_{j-1,L}}\int_{I_{j-1,L}}\,\,\d t\, \d s\,\,H(t-s, \omega(t)-\omega(s))\bigg]\,\P(\d\xi_{j-1,L}),\\
&k(\xi_{j-1,L},\xi_{j,L})
=2\alpha\int_{I_{j-1,L}} \d s \int_{I_{j,L}}\d t \,H(t-s, \omega(t)-\omega(s)) \quad\mbox{and}\\
& \mathscr Z_{L,T}=\frac{Z_{L,T}}{ Z_L^n }
\end{aligned}
\end{equation}
and $Z_L$ is the normalizing constant that makes $\pi_L$ a probability measure.

\noindent{\bf Step 2:}  We now claim that 
\begin{equation}\label{eq:Q4.5}
\lim_{T\to\infty} \|\widehat{\mathbb Q}_T-\widehat{\mathbb Q}_{L,T}\|_{\mathrm{TV}}=0.
\end{equation}
We will prove the above fact using the relative entropy estimate \eqref{eq3:Pinsker}, stated in Lemma \ref{lemma:Pinsker}.  Then 
by comparing 
\eqref{eq:Q.5} and \eqref{eq:Q3} in the above construction, we have an estimate on the relative 
entropy 
\begin{equation}\label{eq:Q:4.55}
\mathrm{Ent}(\widehat{\mathbb Q}_T| \widehat{\mathbb Q}_{L,T})\leq \log\bigg(\frac{Z_{L,T}}{Z_T}\bigg)+  \E^{\widehat{\mathbb Q}_T}\bigg[\sum_{\mathscr I_T} \int\int_{I_{j,L}\times I_{r,L}}\,\d s\, \d t H(t-s, \omega(t)-\omega(s))\bigg] 
\end{equation}

The second term on the right hand side above can be estimated as 
\begin{equation}\label{eq:Q:4.6}
\begin{aligned}
&\E^{\widehat{\mathbb Q}_T}\bigg[\sum_{\mathscr I_T} \int\int_{I_{j,L}\times I_{r,L}}\,\d s\, \d t H(t-s, \omega(t)-\omega(s))\bigg] \\
&\stackrel{\eqref{eq:assumeA}}{\leq} C \sum_{(I_{j,L},I_{r,L})\in\mathscr I_T} \int\int_{I_{j,L}\times I_{r,L}}\,\frac{\d s\, \d t}{(1+|t-s|)^{2+\eps}} \\
&\leq C \frac 1 {L^{2+\eps}} \sum_{(I_{j,L},I_{r,L})\in\mathscr I_T} \int\int_{I_{j,L}\times I_{r,L}}\,\d s\, \d t \leq C_1 \frac {T^2} {L^{2+\eps}}\to 0, 
\end{aligned}
\end{equation}
where for the last argument we used $\#\mathscr I_T=O(n^2)$ and chose, for instance, $L=L(T)=T/\log T$. The first term on the right hand side in \eqref{eq:Q:4.55} can be estimated similarly. 
Indeed, the normalizing constant $Z_T$ in \eqref{eq:Q.5} can be bounded below as
$$
\begin{aligned}
Z_T &= \E^\P\bigg[\exp\bigg\{\alpha\sum_{j=1}^{n-1}\bigg(\int_{I_{j-1,L}}\int_{I_{j-1,L}} \d s \d t H(t-s, \omega(t)-\omega(s)) \\
&\qquad\qquad +  2\int_{I_{j-1,L}}\int_{I_{j,L}} \d s \d t H(t-s, \omega(t)-\omega(s))\bigg) \\
&\qquad\qquad\qquad+\sum_{(I_{j,L},I_{r,L})\in\mathscr I_T} \int\int_{I_{j,L}\times I_{r,L}}\,\d s\, \d t H(t-s, \omega(t)-\omega(s)) \bigg\}\bigg] \\
&\stackrel{\eqref{eq:assumeA}}{\geq} Z_{L,T}\,\, \exp\bigg[-C_1 \sum_{\mathscr I_T} \int\int_{I_{j,L}\times I_{r,L}}\,\frac{\d s\, \d t}{(1+|t-s|)^{2+\eps}}\bigg] 
\end{aligned}
$$
Again by the same argument as \eqref{eq:Q:4.6}, we then have, 
\begin{equation}\label{eq:Q:4.7}
 \log\bigg(\frac{Z_{L,T}}{Z_T}\bigg) \leq  C \sum_{\mathscr I_T} \int\int_{I_{j,L}\times I_{r,L}}\,\frac{\d s\, \d t}{(1+|t-s|)^{2+\eps}} \to 0 
\end{equation}
as $T\to\infty$. Combining \eqref{eq:Q:4.6} and \eqref{eq:Q:4.7} we then have the desired claim \eqref{eq:Q4.5}.

\noindent{\bf Step 3:}  Given \eqref{eq:Q4.5}, the central limit theorem for the rescaled increment process under $\widehat{\mathbb Q}_T$ amounts to proving the same under the measure $\widehat{\mathbb Q}_{L,T}$. For this purpose, 
we will follow the same approach which is developed in Section \ref{sec-Markovian}.
 In this framework, we only need to check that the assertions in Proposition \ref{lemma-HS} and in Lemma \ref{lemma-supinf} hold now {\it{uniformly}} in $L$. For
for the first estimate (as in Proposition \ref{lemma-HS}) we then need to show that 
\begin{equation}\label{lemma-HS-L}
\sup_L \bigg[\int\int \e^{2k (\xi_{j-1,L},\xi_{j,L})} \pi_L(\d\xi_{j-1,L})\pi_L(\d\xi_{j,L})\bigg]<\infty. 
\end{equation}
 Recall the definition of the kernel $k(\xi_{j-1,L},\xi_{j,L})$ in \eqref{Q4}.  Then any ``off-diagonal term"  can be uniformly estimated as 
 \begin{equation}\label{Q5}
 \begin{aligned}
\sup_L \int_{I_{j-1,L}} \d s \int_{I_{j,L}}\d t \,\,H(t-s, \omega(t)-\omega(s)) 
\stackrel{\eqref{eq:assumeA}}{\leq} C\sup_L \int_{(j-1)L}^{jL}\int_{jL}^{(j+1)L}\frac{\d t \, \d s }{(1+|t-s|)^{2+\eps}} 
\leq   C^\prime
 \end{aligned}
 \end{equation}
 which also implies that 
 $$
\sup_L \sup_{\xi,\xi}\int  \e^{2k (\xi_{j-1,L},\xi_{j,L})} \pi_L(\d\xi_{j-1,L})\pi_L(\d\xi_{j,L}) \leq \e^{C^\prime} \int  \int\,\, \pi_L(\d\xi_{j-1,L})\pi_L(\d\xi_{j,L}) = \e^{C^\prime}, 
 $$
 justifying the validity of \eqref{lemma-HS-L} and proving that $\sup_L \|\mathscr L\|_{\mathrm{HS}} <\infty$ where $\mathscr L$ is now the integral operator 
  corresponding to the kernel $\e^{k(\cdot,\cdot)}$ with $k(\cdot,\cdot)$ defined  in \eqref{Q4}. 
  
The second statement similar to Lemma \ref{lemma-supinf} requires showing existence of constants $C_1, C_2\in(0,\infty)$, such that 
\begin{equation}\label{eq-supinf-L}
\sup_L \, \sup_{\xi,\xi^\prime} \e^{k(\xi,\xi^\prime)}  \leq C_1 \qquad\mbox{and } \,\,\,\,   \inf_L \,\, \inf_{\xi,\xi^\prime} \e^{k(\xi,\xi^\prime)} \geq C_2.
\end{equation}

The upper bound again follows from the same estimate as \eqref{Q5}, while for the lower bound we again invoke $H(t,x)\geq - C (1+|t|)^{2+\eps}$ and use integrability of the latter lower bound over $\int_{I_{j-1,L}} \d s \int_{I_{j,L}}\d t$ uniformly in $L$. Given \eqref{lemma-HS-L} and \eqref{eq-supinf-L}, now we can again repeat the arguments of Lemma \ref{lemma-Perron-Frobenius} to obtain eigenvalues $\lambda_0$ and the associated eigenfunction $\psi_0$ which is strictly positive and bounded away from zero and infinity, uniformly in $L$, which also concludes the proof of Theorem \ref{thm2} under Assumption \ref{assumeA}. 
   \qed

\begin{remark}\label{rmk-Coulomb}
When the interaction potential $V$ is singular (as in Theorem \ref{thm2}) and $\rho$ has long-range dependence (e.g. when $H(t,x)=\e^{-|t|}|x|^{-1}$ in $d=3$ corresponds to the Fr\"ohlich Polaron) it seems conceivable to follow the strategy of Section \ref{sec:proof:assumeA} above. For these cases, instead of estimating $V$ by $\|V\|_\infty$, one might appeal to the estimates of singular functionals w.r.t. the uniform norm $\|\Lambda_V\|_\infty$ obtained in the proof of Proposition \ref{lemma-HS} and Lemma \ref{lemma:Lambda}. 
\end{remark}

 \subsection{\bf Proof of Theorem \ref{thm1}} Let  
$$
\chi_\eps(t,x)=\psi_\eps(t) \phi_\eps(x)= \eps^{-(d+2)} \psi(\eps^{-2} t) \phi(\eps^{-1} x).
$$
 Then 
 $$ 
 \dot B_\eps(t,x)= \int_0^t \int_{\R^d} \chi_\eps(t-s,y-x) \dot B(s,y) \d s \, \d y
 $$ is a Gaussian process with covariance 
 $$
 \mathbf E[\dot B_\eps(t,x)\dot B_\eps(s,y)]= (\psi_\eps\star \psi_\eps)(t-s)\,(\phi_\eps\star\phi_\eps)(x-y).
 $$
Let us define the It\^o integral
$$
M_{\eps,t}(W)=\int_0^t\int_{\R^d} \chi_\eps(t-s,W_s-x) \, \dot B(s,x) \d x \d s,
$$
so that  the earlier remark implies that for any two independent Brownian paths $W$ and $W^\prime$, 
$$
\mathbf E\big[M_{\eps,t}(W)M_{\eps,t}(W^\prime)\big]= \int_0^t\int_0^t (\psi_\eps\star\psi_\eps)(\sigma-s) \, (\phi_\eps\star\phi_\eps)(W_\sigma-W^\prime_s) \, \d\sigma\d s
$$
Therefore, 
\begin{equation}\label{eq0:thm1}
\begin{aligned}
&\mathbf E\big[\exp\big\{\beta \eps^{(d-2)/2} M_{\eps,t}(W)\big\}\big]\\
&=\exp\bigg\{ \frac {\beta^2 \eps^{d-2}} 2  \int_0^t\int_0^t (\psi_\eps\star\psi_\eps)(\sigma-s) \, (\phi_\eps\star\phi_\eps)(W_\sigma-W_s) \, \d\sigma\d s\bigg\} \\
&=\exp\bigg\{ \frac{\beta^2 \eps^{-4}} 2   \int_0^t\int_0^t (\psi\star\psi)(\eps^{-2}\sigma-\eps^{-2}s) \, (\phi\star\phi)(\eps^{-1}W_\sigma-\eps^{-1}W_s) \, \d\sigma\d s\bigg\}\\
&\stackrel{\mathrm{(d)}}{=}\exp\bigg\{ \frac{\beta^2}2   \int_0^{t/\eps^2}\int_0^{t/\eps^2} (\psi\star\psi)(\sigma-s) \, (\phi\star\phi)(W_\sigma-W_s) \, \d\sigma\d s\bigg\}
\end{aligned}
\end{equation}

Let us define the {\it{annealed polymer path measure}} as 
\begin{equation}\label{polym-annealed}
\begin{aligned}
\overline{\mathbb Q}_{\beta,\eps,t}(\d W)=\frac 1 {Z_{\beta,\eps,t}} \mathbf E\Big[\exp\big\{\beta \eps^{(d-2)/2}\,\, M_{\eps,t}(W)\big\}\Big]\ \,\,\d\P_0(\d W) 
\end{aligned}
\end{equation}
where $Z_{\beta,\eps,t}=[\mathbf E\otimes \E_0][\exp\big\{\beta \eps^{(d-2)/2}\,\, M_{\eps,t}(W)\big\}]$ is the averaged polymer partition function.
Then Theorem \ref{thm2} (e.g. by Assumption \ref{assumeB} for the compactly supported function $\rho=\psi\star \psi$ and the bounded function $V=\phi\star \phi$) implies that, for any fixed $\beta>0$ and $t>0$
\begin{equation}\label{annealedCLT}
\overline{\mathbb Q}_{\beta,\eps,t} \big[\eps W_{t\eps^{-2}}\in \cdot\big] \Rightarrow \mathbf N_d\big(0,\sigma^2(\beta)\mathbf I_{d\times d}\big) \qquad \mbox{as }\,\,\eps\to 0.
\end{equation}
 To conclude the proof of Theorem \ref{thm1}, note that $\hat u_\eps(t,x)=u_1(t/\eps^2,x/\eps)$ satisfies the equation 
$$
\partial_t \hat u_\eps= \frac 12 \Delta \hat u_\eps+ \beta \eps^{-2} \dot B_1(t/\eps^2,x/\eps) \hat u_\eps \qquad \hat u_\eps(0,x)=u_0(x)
$$
for any continuous and bounded function $u_0:\R^d\to \R$. Therefore, Feynman-Kac formula and the identity \eqref{eq0:thm1} imply that 
$$
\mathbf E[\hat u_\eps(t,x)]= \E_x\bigg[u_0(\eps W_{t/\eps^2}) \exp\bigg\{\frac{\beta^2}2   \int_0^{t/\eps^2}\int_0^{t/\eps^2} (\psi\star\psi)(\sigma-s) \, (\phi\star\phi)(W_\sigma-W_s) \, \d\sigma\d s\bigg\}\bigg], 
$$
while
$$
Z_{\beta,\eps,t}= \E_x\bigg[\exp\bigg\{\frac{\beta^2}2   \int_0^{t/\eps^2}\int_0^{t/\eps^2} (\psi\star\psi)(\sigma-s) \, (\phi\star\phi)(W_\sigma-W_s) \, \d\sigma\d s\bigg\}\bigg]. 
$$
Recall \eqref{eq:step1}, which in the present context implies that for any fixed $\beta>0$ and $\eps>0$, $Z_{\beta,\eps,t}= \exp\big[\frac t{\eps^2} \theta_0+ \theta_1+ o(1)\big]$ as $\eps\to 0$ for some constants $\theta_0, \theta_1$. The last three assertions,  combined with \eqref{annealedCLT} now imply  Theorem \ref{thm1}. 
\qed

\section{Proof of Lemma \ref{lemma:Lambda}}\label{sec:lemma:Lambda}

Recall that for any function $V:\R^d\to \R$ we denote by $\Lambda_1(x)= \int_0^1 V(W_s-x) \, \d s$ where $(W_s)_{s\geq 0}$ is a $d$-dimensional Brownian motion. 
The law of $(W_s)_{s\geq 0}$ starting at $y\in \R^d$ is denoted by $\P_y$. $\E_y$ stands for the corresponding expectation.

Also, note that we need to prove Lemma \ref{lemma:Lambda} with a function $V$ satisfying Assumption \ref{assumeB}.
Since Lemma \ref{lemma:Lambda} follows immediately if $V$ is bounded, so we only need to handle the case of
 the Coulomb potential $V(x)=\frac 1 {|x|^p}$ in $d\geq 3$ for $p\in (0,\frac 2 {d-2})$ as well as for the Dirac potential $V(x)=\delta_0(x)$ for $d=1$. 
 The proof of Lemma \ref{lemma:Lambda} is based on two important ingredients. The first key estimate is 
\begin{lemma}\label{lemma:M}
Fix any function $V$ satisfying Assumption \ref{assumeB}. Then there exist $a\in (0,1)$, $\rho>1$ and $\alpha>0$ such that  with 
\begin{equation}\label{Mdef}
M: = \int\int_{|x_1-x_2|\leq 1}\d x_1\d x_2\, \bigg[\exp\bigg\{\alpha\bigg(\frac{|\Lambda_1(x_1)-\Lambda_1(x_2)|}{|x_1-x_2|^a}\bigg)^\rho\bigg\}- 1 \bigg],
\end{equation}
we have 
\begin{equation}\label{prooflemma3est2.57}
\E_0(M) <\infty.
\end{equation}
 \end{lemma}
\begin{remark}\label{remark:M}
For the above statement, we need to choose
\begin{itemize}
\item For $d\geq 3$, $V(x)=\frac 1 {|x|^p} $ with $p\in (0,\frac 2 {d-2})$:
$$
\begin{aligned}
 &a=2-p - 2\eps \in (0,1) \,\,\mbox{and } \,\, \rho=\frac 1 {1-(1-\delta)\eps},\quad \mbox{with } \\
& \delta\in \bigg(0,\frac{\frac{p+1}d- \frac p2}{1-\frac p 2}\bigg), \quad \eps\in \bigg(\frac{1-\frac {p+1}d}{1-\delta}, 1- \frac p 2\bigg)
\end{aligned}
$$
\item For $d=1$ and $V(x)=\delta_0(x)$,
$$
a=1-2\eps \in (0,1) \,\,\,\, \mbox{and } \,\,\,\, \rho=\frac 1 {1-\eps}>1, \quad\,\,\mbox{with }\,\, \eps\in \bigg(\frac 14,\frac1 2\bigg) 
$$
\end{itemize}
\end{remark}\qed


The second ingredient needed for the proof of Lemma \ref{lemma:Lambda} is the following multidimensional version of Garsia-Rodemich-Rumsey estimate \cite[p.\,60]{SV79}.

\begin{lemma}\label{GRR}
Let $q(\cdot)$ and $\Psi(\cdot)$ be strictly increasing continuous functions on $[0,\infty)$ so that $q(0)=\Psi(0)=0$ and $\lim_{t\uparrow\infty} \Psi(t)=\infty$. If $f\colon \R^d\to \R$ is continuous on the closure of the ball $B_{2r}(z)$ for some $z\in \R^d$ and $r>0$, then the bound 
\begin{equation}\label{GRR1}
\int_{B_r(z)} \d x \int_{B_r(z)}  \d y\,\, \Psi\bigg(\frac{|f(x)-f(y)|}{q(|x-y|)}\bigg) \leq M<\infty,
\end{equation}
implies that
\begin{equation}\label{GRR2}
\big|f(x)- f(y)\big| \leq 8 \int_0^{2|x-y|} \Psi^{-1}\bigg(\frac{M}{\gamma u^{2d}}\bigg) \, q(\d u), \qquad x,y\in B_r(z),
\end{equation}
for some constant $\gamma$ that depends only on $d$.
\end{lemma}
We will complete the proof of Lemma \ref{lemma:Lambda} with the help of the above two results in Section \ref{subsec:lemma:Lambda}. Let us now turn to the proof of Lemma \ref{lemma:M}.

\subsection{{\bf{Proof of Lemma \ref{lemma:M}}}}

\subsubsection {\bf Proof of Lemma \ref{lemma:M} for $d\geq 3$}

We will first prove the lemma when $V(x)=\frac 1 {|x|^p}$ if $p\in (0,2/(d-2))$ and $d\geq 3$. 
We will have several occasions to use the following simple inequality in the proof of Lemma \ref{lemma:M}.

\begin{lemma}\label{lemma:bound}
Let $b>a>0$. If $p\geq 1$, then 
$$
p a^{p-1} (b-a) \leq b^p -a^p \leq p b^{p-1} (b-a).
$$
If $p\in (0,1)$, the reverse inequality holds.\end{lemma}
\begin{proof}
By the mean-value theorem, $b^p-a^p=(b-a) g^\prime(\xi)$ for some $\xi\in(a,b)$ where $g(x)=x^p$. Now if $p\geq 1$, then $g^\prime$ is increasing, meaning $g^\prime(\xi)\in [g^\prime(a),g^\prime(b)]$ implying the requisite inequality. If $p\in (0,1)$, then $g^\prime$ is decreasing, which reverses the inequality.
\end{proof}

The proof of Lemma \ref{lemma:M} in $d\geq 3$ now splits into two main tasks. 
First, we want to show that for 
any $\delta$ and $\eps$ as in Remark \ref{remark:M} (for $d\geq 3$) and $a=2-p- 2\eps$ and $\rho= \frac 1{1-(1-\delta)\eps}$,  
\begin{equation}\label{eqlemma1}
\sup_{\heap{x_1,x_2\in\R^d}{|x_1-x_2|\leq 1}} \sup_{x\in \R^d}\E_x \bigg[\exp\bigg\{\alpha \bigg(\frac{\big|\Lambda_1 (x_1)-\Lambda_1 (x_2)\big|}{|x_1-x_2|^{a}}\bigg)^\rho\bigg\}\bigg] <\infty.
\end{equation}
for some $\alpha\in (0,\infty)$. Let us first assume \eqref{eqlemma1} and conclude 

\medskip

\noindent{\bf Proof of Lemma \ref{lemma:M} for $d\geq 3$ (Assuming \eqref{eqlemma1}):} It suffices to show that 
there exists a constant $\alpha_1=\alpha_1(\eps)>0$ such that the random variable
\begin{equation}\label{Mdef2}
M=\int_{\R^d}\d x_1 \int_{\R^d}\d x_2\,\1\{|x_1- x_2|\leq 1\} \,\, \bigg[\exp\bigg\{\alpha_1 \bigg(\frac{\big|\Lambda_1(x_1)-\Lambda_1(x_2)\big|}{|x_1-x_2|^{a}}\bigg)^\rho\bigg\}-1\bigg]
\end{equation}
has a finite expectation under $\P_0$. For this purpose, by Fubini's theorem it suffices to show that 
\begin{equation}\label{claim2}
\int\int_{|x_1- x_2|\leq 1} \d x_1 \d x_2 \,\, \E\bigg[\exp\bigg\{\alpha_1 \bigg(\frac{\big|\Lambda_1(x_1)-\Lambda_1(x_2)\big|}{|x_1-x_2|^{a}}\bigg)^\rho\bigg\}-1\bigg] < \infty.
\end{equation}
We will now prove \eqref{claim2}. We can write 
$$
\R^d\subset \bigcup_{n=0}^\infty\big\{x\in\R^d\colon n\leq |x|< n+1\big\}
$$
and for some $\theta\in (0,1)$ to be chosen later, we put $\tau_n= \inf\{t>0\colon\, |W_t| > n- n^\theta\}$.  
Then 
\begin{equation}\label{taubigless1}
\begin{aligned}
&\int\int_{|x_1- x_2|\leq 1} \d x_1 \d x_2 \,\, \E\bigg[\exp\bigg\{\alpha_1 \bigg(\frac{\big|\Lambda_1(x_1)-\Lambda_1(x_2)\big|}{|x_1-x_2|^{a}}\bigg)^\rho\bigg\}-1\bigg] \\
&\leq \sum_{n=0}^\infty \int_{|x_1|\in [n,n+1)} \d x_1 \int_{B_1(x_1)} \d x_2 \,\, \bigg[ \E\bigg\{\1_{\{\tau_n>1\}} \bigg(\exp\bigg\{\alpha_1 \bigg(\frac{\big|\Lambda_1(x_1)-\Lambda_1(x_2)\big|}{|x_1-x_2|^{a}}\bigg)^\rho\bigg\}-1\bigg)\bigg\}
\\
&\qquad\qquad\qquad\qquad\qquad+\E\bigg\{\1_{\{\tau_n\leq1\}} \bigg(\exp\bigg\{\alpha_1 \bigg(\frac{\big|\Lambda_1(x_1)-\Lambda_1(x_2)\big|}{|x_1-x_2|^{a}}\bigg)^\rho\bigg\}-1\bigg\}\bigg)\bigg].
\end{aligned}
\end{equation}
To estimate the first expectation, we observe that for $|x_1| \in [n,n+1)$  and for $|x_2-x_1|<1$, if $\tau_n >1$, then $|W_s- x_1| > n^\theta$ and $|W_s- x_2| > n^\theta - 1$ for any $s\in[0,1]$. 
Recall that for $V(x)=\frac 1 {|x|^p}$, 
$$
\big|\Lambda(x_1)-\Lambda(x_2)\big| =\bigg| \int_0^1 \d s \, \bigg(\frac 1 {|W_s-x_1|^p} - \frac 1 {|W_s-x_2|^p}\bigg) \bigg| 
$$
Then in order to estimate the above quantity, by Lemma \ref{lemma:bound} and triangle inequality, it follows that 
for any $n\in\N$, on the event $\{\tau_n>1\}$,
$$ 
\begin{aligned}
\frac{\big|\Lambda_1(x_1)-\Lambda_1(x_2)\big|}{|x_1-x_2|^{a}} &\leq C \frac {|x_1-x_2|}{|x_1-x_2|^{2-p-2\eps}} \, \int_0^1 {\d s}\bigg(\frac 1{|W_s -x_1| ^p |W_s- x_2|}+ \frac 1 {|W_s -x_1|  |W_s- x_2|^p}\bigg)\\
&\leq C |x_1-x_2|^{p-1+2\eps} n^{-(p+1)\theta}\\
&\leq C  n^{-(p+1)\theta}.
\end{aligned}
$$
We remark that in the first upper bound above we also used the choice $a=2-p-2\eps$ as in Remark \ref{remark:M}. There we also chose $\delta>0$ and $\eps>\frac {1- \frac{p+1}d}{1-\delta}> 1- \frac{p+1}d$, so that  $p-1 + 2\eps >(p+1)(1-\frac 2d)\geq 0$ in $d\geq 3$, which justifies the third upper bound.

Then the last upper bound dictates that  
\begin{equation}\label{taubig1}
\begin{aligned}
&\sum_{n=0}^\infty \int_{|x_1|\in [n,n+1)} \d x_1 \int_{B_1(x_1)} \d x_2 \,\, \E\bigg\{\1_{\{\tau_n>1\}} \bigg(\exp\bigg\{\alpha_1 \bigg(\frac{\big|\Lambda_1(x_1)-\Lambda_1(x_2)\big|}{|x_1-x_2|^{a}}\bigg)^\rho\bigg\}-1\bigg)\bigg\} \\
&\leq C \sum_{n=0}^\infty \bigg(\e^{\alpha_1 C^\rho n^{-(p+1)\theta\rho}} - 1\bigg) |A_{n,n+1}| 
\end{aligned}
\end{equation}
where $A_{n,n+1}$ denotes the annulus $B_{n+1}(0)\setminus B_n(0)\subset \R^d$. In the above summand since the first term is of size $O(n^{-(p+1)\theta\rho})$ and $|A_{n,n+1}|=O(n^{d-1})$, the above sum is finite if 
$\theta> \frac d{(p+1)\rho}$. 
Recall Remark \ref{remark:M} where we chose $\eps>\frac {1- \frac{p+1}d}{1-\delta}$ and $\rho=\frac 1 {1-(1-\delta)\eps}$. Therefore, $1-(1-\delta)\eps<\frac{p+1}d$ and hence $\rho= \frac 1 {1-(1-\delta)\eps} > \frac d{p+1}$, so that we can choose $\theta\in (0,1)$ with $\theta > \frac d {(p+1)\rho}$, as desired. 

Now the second expectation in \eqref{taubigless1} can be bounded by the Cauchy-Schwarz inequality and \eqref{eqlemma1} so that 
for a suitable $\alpha_1$, any $x_1,x_2\in\R^d$ with $|x_1-x_2|\leq1$, we have 
$$
\begin{aligned}
\E\bigg[\1_{\{\tau_n\leq1\}} &\bigg\{\exp\bigg\{\alpha_1 \bigg(\frac{\big|\Lambda_1(x_1)-\Lambda_1(x_2)\big|}{|x_1-x_2|^{a}}\bigg)^\rho\bigg\}-1\bigg\}\bigg] \\
&\leq \P\big(\tau_n\leq 1\big)^{\frac 12} \,\, \E\bigg[\exp\bigg\{2\alpha_1 \bigg(\frac{\big|\Lambda_1(x_1)-\Lambda_1(x_2)\big|}{|x_1-x_2|^{a}}\bigg)^\rho\bigg\}\bigg]^{\frac 12} \\
&\leq C \P\bigg(\max_{[0,1]} |W| > n- n^\theta\bigg)^{\frac 12},
\end{aligned}
$$
where $C$ does not depend on $x_1,x_2$. Since the last probability is of order $\e^{-c n^2}$, the second sum on $n$ in \eqref{taubigless1} is obviously finite. This, combined with the finiteness of the sum in \eqref{taubig1}, proves \eqref{claim2}. Lemma \ref{lemma:M} is therefore also proved, assuming the estimate \eqref{eqlemma1}. 
\qed

\medskip 

It remains to provide the 

\noindent{\bf Proof of \eqref{eqlemma1}:}  We now fix $x_1, x_2\in \R^3$ with $|x_1-x_2|\leq 1$ and denote
$$
\mathscr V_{x_1,x_2}(y)= \frac 1{|y-x_1|^p}- \frac 1{|y-x_2|^p} 
$$
For any $x_1,x_2$ satisfying $|x_1-x_2|\leq 1$, and $a=2-p-2\eps\in (0,1)$, we estimate, again by Lemma \ref{lemma:bound}, 
\begin{equation}\label{Vhatesti}
\begin{aligned}
|\mathscr V_{x_1,x_2}(y)|
&=\frac{\big| |y-x_2|^p-|y-x_1|^p\big|}{|y-x_1|^p\,|y-x_2|^p} \\
&\leq C {|x_1-x_2|}\bigg[\frac 1 {|y-x_1|^p\,|y-x_2|}+ \frac 1 {|y-x_1|^p\,|y-x_2|}\bigg]\\
&\leq |x_1-x_2|^a\,\,  \big[|y-x_2|^{1-a}+|y-x_1|^{1-a}\big] \bigg[\frac 1 {|y-x_1|^p\,|y-x_2|}+ \frac 1 {|y-x_1|\,|y-x_2|^p}\bigg]
 \end{aligned}
\end{equation}
where we have used $(r+s)^{1-a}\leq r^{1-a}+s^{1-a}$ for any $r,s\geq0$. In order to prove \eqref{eqlemma1} using the above estimate, we need a suitable upper bound on the 
truncated Greens's function. Let us denote by $\eta=|x-y|^2\wedge 1$ and $z=\frac {|y-x|^2}{2t}>0$ note that, for any $b>0$, the map $w\mapsto w^{d/2+b} \e^{-w}$ is bounded. Then, 
$$
\begin{aligned}
\int_0^{\eta}\d t\,\frac{\e^{-|y-x|^2/2t}}{t^{d/2}} 
&\leq C|y-x|^{-d-2b}\int_0^{\eta}\d t t^b\\
&\leq C |y-x|^{-d-2b}\big(|y-x|^2\wedge 1\big)^{1+b},\qquad x,y\in\R^d.
\end{aligned}
$$
and 
$$
\int_{\eta}^1 \d t\,\frac{\e^{-|y-x|^2/2t}}{t^{d/2}} \leq\int_{\eta}^1 \d t\,t^{-d/2}
\leq \big[|y-x|^2\wedge 1\big]^{1-d/2}-1,
$$
implying  
\begin{equation}\label{trankernelesti}
\int_{0}^1 \d t\,\frac{\e^{-|y-x|^2/2t}}{t^{d/2}} \leq C \frac 1{|y-x|^{d-2}(1+|y-x|)^b}.
\end{equation}
Then \eqref{Vhatesti}, the above bound as well as symmetry in $x_1$ and $x_2$ ensure that 
\begin{equation}\label{eqlemma2}
\sup_{\heap{x\in \R^d}{|x_1-x_2| \leq 1}}  \,\, \E_x\bigg[\int_0^1 \frac{|V(W_s)|}{|x_1-x_2|^{a\rho}} ^\rho\,\d s\bigg] <\infty
\end{equation}
provided we show that 
$$
\sup_{\heap{x_1,x_2\in\R^d\colon}{|x_1-x_2|\leq 1}}\sup_{x\in\R^d}\int_{\R^d}\,\frac {\d y}{(1+|y-x|)^b}\,\frac1{|y-x_1|^{p\rho}}\times\frac 1{|y-x|^{d-2}}\times\frac1{|y-x_2|^{\rho a}}<\infty.
$$
We now choose suitable constants $\gamma_i>1$ such that $\sum_{i=1}^3 \gamma_i^{-1}=1$ and apply H\"older's inequality w.r.t. the measure $\frac {\d y}{(1+|y-x|)^b}$. It turns out that
$$
\begin{aligned}
I_1:=\int_{\R^d} \frac{\d y}{(1+|y|)^b} \frac 1 {|y|^{\rho p \gamma_1}} &=\int_0^\infty \frac {\d r}{(1+r)^b} \, \frac {r^{d-1}} {r^{\rho p \gamma_1}} \\
&\leq \int_0^1 \frac{\d r} {r^{\rho p \gamma_1- d +1}} + \int_1^\infty \frac{\d r}{(1+r)^b \, r^{\rho p \gamma_1- d +1}}.
\end{aligned}
$$
Note that the second integral is finite provided we choose $b>d$ and the first integral is finite if $\rho p \gamma_1- d+ 1 <1$, i.e. if we choose 
\begin{equation}\label{gamma1}
\gamma_1< \frac d {\rho p}.
\end{equation}
Likewise, 
$$
\begin{aligned}
I_2&:=\int_0^\infty \frac {\d r}{(1+r)^b} \, \frac {r^{d-1}} {r^{(d-2) \gamma_2}} \\
&\leq \int_0^1 \frac{\d r} {r^{(d-2) \gamma_2- d +1}} + \int_1^\infty \frac{\d r}{(1+r)^b \, r^{(d-2) \gamma_2- d +1}},
\end{aligned}
$$
and again the second integral is finite provided we choose $b>d$ and the first integral is finite if we choose 
\begin{equation}\label{gamma2}
\gamma_2< \frac d {d-2}
\end{equation}
Finally, 
$$
\begin{aligned}
I_3&:=\int_0^\infty \frac {\d r}{(1+r)^b} \, \frac {r^{d-1}} {r^{\rho a \gamma_3}} \\
&\leq \int_0^1 \frac{\d r} {r^{\rho a \gamma_3- d +1}} + \int_1^\infty \frac{\d r}{(1+r)^b \, r^{\rho a \gamma_3- d +1}}
\end{aligned}
$$
and with $b<d$ the second integral is finite  and the first integral is finite if we choose 
\begin{equation}\label{gamma3}
\gamma_3< \frac d {\rho a}
\end{equation}
We choose $\gamma_1>1$ and $\gamma_2>1$ satisfying \eqref{gamma1}-\eqref{gamma2}. Since $\frac 1 {\gamma_3}=1- \frac 1 {\gamma_1}-\frac 1 {\gamma_2}$, we must have 
$\frac 1 {\gamma_3} <1- \frac{\rho p}d - \frac {d-2} d$ and together with \eqref{gamma3}, we must have $\rho< \frac 2 {p+a}$. But with our choice $\rho=\frac 1 {1-(1-\delta)\eps}$ and $a= 2-p - 2\eps$ (recall Remark \ref{remark:M}), the desired requirement $\rho< \frac 2 {p+a}$ is met. Hence, the proof of  \eqref{eqlemma2} is completed.\footnote{The above argument for showing finiteness of $I_1, I_2, I_3$ resembles an idea similar to the proof of \cite[Lemma 2.1, p.2222]{KM15} for the particular case $V(x)=1/|x|$ in $d=3$, which however has a minor gap stemming from an erroneous application of the H\"older's inequality.}

Finally we can choose some $\alpha\in (0,\infty)$ (sufficiently small, if necessary), so that by Jensen's inequality, 
$$
\sup_{\heap{x\in \R^d}{|x_1-x_2| \leq 1}}  \,\, \E_x\bigg[\alpha \bigg(\frac{\big|\Lambda_1(x_1)-\Lambda_1(x_2)|}{|x_1-x_2|^a}\bigg)^\rho\,\bigg] =\kappa<1, 
$$
and subsequently, by Khas'misnki's lemma again, 
$$
\sup_{\heap{x\in \R^d}{|x_1-x_2| \leq 1}}  \,\, \E_x\bigg[\exp\bigg\{\alpha \bigg(\frac{\big|\Lambda_1(x_1)-\Lambda_1(x_2)|}{|x_1-x_2|^a}\bigg)^\rho\bigg\}\bigg] \leq \frac 1 {1-\kappa}<\infty. 
$$
Hence, \eqref{eqlemma1} is proved $V(x)=\frac 1 {|x|^p}$ for $p\in (0,2/(d-2))$ and $d\geq 3$. 
\qed

\subsubsection{{\bf Proof of Lemma \ref{lemma:M} for $d=1$}}

 Note that for the desired estimate \eqref{eqlemma1} for $V(x)=\delta_0(x)$ in $d=1$ we are no longer allowed to invoke interpolation and  triangle inequality as we did in \eqref{Vhatesti}, nor can we apply Khas'misnki's lemma. Instead, we will use the symmetry properties of the function 
\begin{equation}\label{eq0Dirac}
\mathscr V_{x_1,x_2}(y)= \delta_{x_1}(y)- \delta_{x_2}(y) \qquad \mbox{ such that }\,\, \mathscr V_{x_1,x_2}(y)=0\quad\mbox{for }\, |y|>h:=\frac 12 |x_1-x_2|. 
\end{equation}
which satisfies 
\begin{equation}\label{eq1Dirac}
\int_\R \mathscr V_{x_1,x_2}(y) \d y=0 \quad \mbox{and }\,\, \int_\R| \mathscr V_{x_1,x_2}(y) | \d y<\infty.
\end{equation}
Using the above symmetry property a technique was developed in \cite{DV77}
 to prove large deviations and law of iterated logarithm for one-dimensional Brownian local times in the uniform metric. 
 For the sake of completeness, we will collect the relevant material from there  
to show the requisite claim \eqref{eqlemma1}. 
\begin{lemma}\label{lemmaDirac}
Let $\mathscr V_{x_1,x_2}$ be defined as in \eqref{eq0Dirac}. Fix $\eps\in (\frac 14,\frac 12)$, $a=1-2\eps$ and $\rho=\frac 1 {1-\eps}$ as in Remark \ref{remark:M}. Then, for some $\alpha>0$, and any $x\in \R$, 
\begin{equation}\label{Diracitem1}
\E_x\bigg[\exp\bigg\{\alpha\bigg| \int_0^1 \d s \, \frac{V_{x_1,x_2}(W_s)}{|x_1-x_2|^{a}}\bigg|^{\rho}\bigg\}\bigg]<\infty.
\end{equation}
and also with $\Lambda_1(x)= \int_0^1 \delta_0(W_s-x)\ d s$
\begin{equation}\label{Diracitem2}
\int\int_{|x_1-x_2|\leq 1} \d x_1 \d x_2 \,\, \bigg[\exp\bigg\{\alpha\bigg(\frac{\Lambda_1(x_1)-\Lambda_1(x_2)}{|x_1-x_2|^a}\bigg)^\rho\bigg\}-1 \bigg] =:M< \infty. 
\end{equation} 
almost surely. 
\end{lemma}
\begin{proof}
We first prove \eqref{Diracitem1}. Note that we can write transition density of the one-dimensional Brownian path as $p_t(x,y)= C t^{-1/2} \mathcal F(t^{-1/2}(x-y))$ with 
$$
\mathcal F(y)= c_0 \int_\R \e^{-\lambda^2} \e^{-\mathbf i \lambda y} \d \lambda 
$$
we have $\|\mathcal F\|_\infty<\infty$ and  for any $a \in (0,1/2]$. 
\begin{equation}\label{eq2Dirac}
\begin{aligned}
|\mathcal F(y)- \mathcal F(0)| \leq c_0 \int_\R \d \lambda \e^{-\lambda^2} |\e^{\mathbf i \lambda y}- 1| 
&= c_1 \int_\R \d \lambda \e^{-\lambda^2} \sqrt{|1-\cos(\lambda y)|} \\
&\leq c_2 \int_\R \d \lambda \e^{-\lambda^2} |\sin(\lambda y/2)| \leq c_3 |y|^{2a}. 
\end{aligned}
\end{equation} 
For any positive integer $n$, let us write $\int_<\d s=\int_{0\leq s_1\leq s_2\leq \dots\leq s_n\leq 1} \d s_1\d s_2\dots \d s_n$. Then, for any $u \in [0,1]$, 
\begin{equation}\label{eq3Dirac}
\begin{aligned}
&\E_x\bigg[\big(\int_0^1 \d s V_{x_1,x_2}(W_s)\big)^{2n}\bigg]
= (2n)! \int_< \d s\int_{\R^{2n}} \d y \prod_{j=0}^{2n-1} \bigg[\frac{ \mathscr V_{x_1,x_2}(y_{j+1})}{(s_{j+1}-s_j)^{1/2}} \, \mathcal F\bigg(\frac{y_{j+1}- y_j}{(s_{j+1}-s_{j})^{1/2}}\bigg)\bigg] \\
&\stackrel{\eqref{eq1Dirac}} {=} (2n)! \int_< \d s\int_{\R^{2n}} \d y \prod_{j=0}^{2n-1} \bigg(\frac{ \mathscr V_{x_1,x_2}(y_{j+1})}{(s_{j+1}-s_j)^{1/2}}\bigg)\\
 &\qquad\qquad\qquad\prod_{k=0}^{n-1}\bigg[\mathcal F\bigg(\frac{y_{2k+1}- y_{2k}}{(s_{2k+1}-s_{2k})^{1/2}}\bigg) \mathcal F\bigg(\frac{y_{2k+2}- y_{2k}}{(s_{2k+2}-s_{2k})^{1/2}}\bigg)
 - \mathcal F^2(0) \bigg]
\end{aligned}
\end{equation}
Now for $|z_1|\leq h$ and $|z_2| \leq h$, we have 
$$
\big| \mathcal F(t_1^{-1/2} z_1) \mathcal F(t_2^{-1/2} z_2)- \mathcal F^2(0)\big| \stackrel{\eqref{eq2Dirac}}{\leq} C\|\mathcal F\|_\infty \big[t_1^{-a} |z_1|^{2a}+ t_2^{-a} |z_2|^{2a}\big] \leq Ch^{2a} \big[t_1^{-a}+ t_2^{-a}\big].
$$

In \eqref{eq3Dirac} we want to 
use the last estimate and the upper bound $\int_\R \d y |\mathscr V_{x_1,x_2}(y)| <\infty$ as well as the fact that $\mathscr V$ vanishes outside $B_h(0)$ (recall \eqref{eq0Dirac}). Then we have 
\begin{equation}\label{eq4Dirac}
\begin{aligned}
&\E_x\bigg[\big|\int_0^1 \d s \mathscr V_{x_1,x_2}(W_s)\big|^{2n}\bigg]\\
&\leq (2n)! C^n h^{2na} \int_< \d s \prod_{j=0}^{2n-1} \bigg(\frac{ 1}{(s_{j+1}-s_j)^{1/2}}\bigg) \, \prod_{k=0}^{n-1}\bigg[\frac 1 {|s_{2k+1}- s_{2k}|^{a}}+ \frac 1 {(s_{2k+2}-s_{2k})^{a}}\bigg] \\
&=(2n)! C^n h^{2na} \int_< \d s \prod_{j=0}^{n-1} \bigg[\bigg(\frac 1 {|s_{2j+1}- s_{2j}|^{a}}+ \frac 1 {(s_{2j+2}-s_{2j})^{a}}\bigg) 
\bigg(\frac{ 1}{(s_{2j+1}-s_{2j})^{1/2}}\bigg) \\
&\qquad\qquad\qquad\qquad\qquad\qquad\qquad\times  \bigg(\frac{ 1}{(s_{2j+2}-s_{2j+1})^{1/2}}\bigg)\bigg]
\end{aligned}
\end{equation}
Let 
$$
I=\int_x^y \bigg[\frac 1 {(z-x)^{a}}+ \frac 1 {(y-z)^{a}}\bigg] \frac 1 {(y-x)^{1/2}} \, \frac 1 {(y-z)^{1/2}} \, \d z.
$$
If we substitute $z= x+ \lambda (y-x)$ which ensures $\lambda\in [0,1]$, we have with $a \in (0,1/2)$ 
\begin{equation}\label{eq5Dirac}
I= \frac 1 {(y-x)^{a}} \int_0^1 \bigg[\frac 1 {\lambda^{a}}+ \frac 1 {(1-\lambda)^{a}}\bigg] \frac 1 {\lambda^{1/2}} \frac 1 {(1-\lambda)^{1/2}} \, \d\lambda \leq C \frac 1 {(y-x)^{a}}
\end{equation}
Moreover, we notice that 
\begin{equation}\label{eq6Dirac}
\int_{0\leq s_2\leq s_4\leq \dots\leq s_{2n}\leq 1} \d s_2\dots \d s_{2n} \prod_{j=0}^{n-1} \frac 1 {(s_{2j+2}- s_{2j})^{a}} = \frac {[\Gamma(1-a)]^n}{\Gamma\big(1+n(1-a)\big)}.
\end{equation} 
We now apply \eqref{eq5Dirac} and \eqref{eq6Dirac} to \eqref{eq4Dirac} to obtain, 
\begin{equation}\label{eq7Dirac}
\begin{aligned}
\E_x\bigg[\big|\int_0^1 \d s \mathscr V_{x_1,x_2}(W_s)\big|^{2n}\bigg] \leq (2n)! C^n h^{2na}  \frac {[\Gamma(1-a)]^n}{\Gamma\big(1+n(1-a)\big)}
\end{aligned}
\end{equation}
where 
$\Gamma(1+k)=\int_0^\infty x^k \e^{-x} \d x = k! \sim (k/\e)^k \sqrt{2\pi k}$. Finally, note that, for any $\rho\in (0,2)$, by Jensen's inequality, $\E[\exp\{\alpha X^\rho\}] \leq \sum_{n=0}^\infty \frac{\alpha^n}{n!} \E[X^{2n}]^{\rho/2}$.  Therefore, by \eqref{eq7Dirac}, and by absorbing $\Gamma(1-a)^n$ in $C^n$ on the right hand side, we have 
$$
\begin{aligned}
&\E_x\bigg[\exp\{\alpha\bigg| \int_0^1 \d s \,\frac{\mathscr V_{x_1,x_2}(W_s)}{|x_1-x_2|^{a}}\bigg|^{\rho}\bigg] \\
&\leq \sum_{n=0}^\infty \frac{\alpha^n}{n!} C^{n\rho/2} \bigg[\frac {(2n)!} {\Gamma\big(1+ n(1-a)\big)}\bigg]^{\rho/2}.
\end{aligned}
$$
Since $\rho=\frac 1 {1-\eps}=\frac 2 {1+a} \in (0,2)$, by using Stirling's formula for $(2n)!$ as well as for $\Gamma\big(1+ n(1-a)\big)$, we can choose $\alpha$ small enough if needed to make the last geometric series convergent, proving \eqref{Diracitem1}.

To prove \eqref{Diracitem2}, again we can decompose $\{x_1,x_2\colon |x_1-x_2| \leq 1\} \subset \cup_{n=-\infty}^\infty \big\{x_1,x_2\colon |x_1-n| \leq 1, \, |x_2- n | \leq 1\big\}$, 
and it suffices to show that 

\begin{equation}\label{Diracitem3}
\sum_{n=-\infty}^\infty \,\, \int\int_{\heap{|x_1-n| \leq 1}{ |x_2-n|\leq 1}} \d x_1 \d x_2 \,\, \E_0\bigg[\exp\bigg\{\alpha\bigg(\frac{\Lambda_1(x_1)-\Lambda_1(x_2)}{|x_1-x_2|^a}\bigg)^\rho\bigg\}-1 \bigg] < \infty.
\end{equation} 

If $\tau_n:=\inf\{t\colon W_t \in [n-1,n+1]\}$, by strong Markov property 

$$
\begin{aligned}
& \E_0\bigg[ \1_{\{\tau_n<1\}}\,\, \exp\bigg\{\alpha\bigg(\frac{\Lambda_1(x_1)-\Lambda_1(x_2)}{|x_1-x_2|^a}\bigg)^\rho\bigg\}-1 \bigg] \\
&= \E_0\bigg[ \1_{\{\tau_n<1\}}\,\, \E_{W_{\tau_n}}\bigg\{\exp\bigg\{\alpha\bigg(\frac{\Lambda_{1-\tau_n}(x_1)-\Lambda_{1-\tau_n}(x_2)}{|x_1-x_2|^a}\bigg)^\rho\bigg\}-1\bigg\} \bigg] \\
\end{aligned}
$$
But by the first estimate \eqref{Diracitem1}, 
$$
\E_{W_{\tau_n}}\bigg\{\exp\bigg\{\alpha\bigg(\frac{\Lambda_{1-\tau_n}(x_1)-\Lambda_{1-\tau_n}(x_2)}{|x_1-x_2|^a}\bigg)^\rho\bigg\}-1\bigg\} \leq C <\infty, 
$$
so that for \eqref{Diracitem3} it suffices to show that $\sum_{n=-\infty}^\infty \P_0[\tau_n < 1 ]< \infty$, which follows easily from standard computation using transition density of one dimensional Brownian motion.

\end{proof}

\subsection{\bf Concluding the proof of Lemma \ref{lemma:Lambda}}\label{subsec:lemma:Lambda} 

 Note that, we can write, for any $\kappa>0$, 
\begin{equation}\label{prooflemma3est2.56}
\|\Lambda_1\|_\infty \leq \sup_{x_1, x_2 \in \R^3\colon |x_1-x_2|\leq\kappa} \big|\Lambda_1(x_1)-\Lambda_1(x_2)\big| + \sup_{x\in \kappa\Z^3} \int_0^1 {\d s} {V(W_s -x)}.
\end{equation}
Let us now handle the first summand, for which we would like to apply Lemma \ref{GRR}. We pick $\delta$, $\eps$, $a$ and $\rho$ as in Remark \ref{remark:M} and choose
\begin{equation}\label{choices}
\Psi(x)= \e^{\beta |x|^\rho}-1, \qquad q(x)= |x|^{a }= |x|^{1-2\eps}, \qquad f(x)=\Lambda(x).
\end{equation}
Then $\Psi(\cdot)$, $q(\cdot)$ and $f(\cdot)$ all satisfy the requirements of Lemma \ref{GRR}. Furthermore, Lemma \ref{lemma:M} gurantees that hypothesis \eqref{GRR1} is satisfied if $|x_1- x_2|\leq \kappa$ and $\kappa>0$ is chosen small enough . Hence, \eqref{GRR2} implies that for any fixed constant $\gamma>0$, 
\begin{equation}\label{estimate:GRR}
\sup_{|x_1-x_2|\leq \kappa} \big|\Lambda_1(x_1)-\Lambda_1(x_2)\big| \leq \frac{(1-2\eps)}{\alpha^{1/\rho}} \int_0^\kappa \log\bigg(1+ \frac M {\gamma u^{2d}}\bigg)^{1/\rho} \, u^{-2\eps} \, \d u
\end{equation}
Now if we choose $\kappa$ small enough, then the right hand side above is smaller than 
$$
\frac{1-2\eps}{\alpha^{1/\rho}} C(\kappa) \log\big(M\vee 1\big)^{1/\rho}
$$
for some constant $C(\kappa)$ which goes to $0$ as $\kappa\to 0$. Hence, for any $C>0$, by \eqref{prooflemma3est2.57}, we have
$$
\E\bigg\{\e^{C\sup_{ |x_1-x_2|\leq\kappa} \big|\Lambda_1(x_1)-\Lambda_1(x_2)\big|}\bigg\}<\infty.
$$

Let us turn to the second term on the right hand side of \eqref{prooflemma3est2.56}. Since we are interested in the behavior 
of the path in the time horizon $[0,1]$, it is enough to estimate the supremum in a bounded box. We will show that, for any fixed $\kappa>0$ and any $C>0$,
\begin{equation}\label{prooflemma3est2.58}
\E\bigg[\sup_{\heap{x\in \kappa\Z^3}{|x\leq 2}}\exp\bigg\{ C\int_0^1  {\d s} V(W_s -x) \bigg\}\bigg] \leq (2/\delta)^3\E\bigg[\exp\bigg\{ C\int_0^1 {\d s} V(W_s )\bigg\}\bigg] <\infty.
\end{equation}
Let us first prove the estimate for $V(x)=\frac 1 {|x|}$ when $d=3$. The other cases will follow a very similar strategy. For any $\eta>0$, we can write $1/|x|= V_\eta(x)+ Y_\eta(x)$ for $V_\eta(x)= 1/(|x|^{2}+\eta^{2})^{1/2}$. Since, for any fixed $\eta>0$, $V_\eta$ is a bounded function, 
the above claim holds with $V_\eta(W_s)$ replacing $1/|W_s|$. Hence, (by Cauchy-Schwarz inequality, for instance), it suffices to check the above statement with the difference $Y_\eta(W_s)$, which can be written as
$$
\begin{aligned}Y_\eta(x)=\frac{1}{|x|}-\frac{1}{\sqrt {\eta^{2}+|x|^{2}}}
&=\frac{\sqrt {\eta^{2}+|x|^{2}}-|x|}{|x|\sqrt{\eta^{2}+|x|^{2}}}=\frac{\eta^2}{|x|+\sqrt {\eta^2+|x|^2}}\,\,\frac{1}{\sqrt {\eta^2+|x|^2}}\,\,\frac{1}{|x|}\\
&=\eta^{-1}\phi\bigg(\frac{x}{\eta}\bigg),
\end{aligned}
$$
with
$$
\phi(x)=\frac{1}{|x|}\,\,\frac 1{\sqrt{1+|x|^2}}\,\, \frac 1{|x|+\sqrt{1+|x|^2}}.
$$
One can bound  $\phi(x)$ by $\frac{b}{|x|^\frac{3}{2}}$, since it behaves like $\frac{1}{|x|}$ near $0$ and like $\frac{1}{|x|^3}$ near $\infty$. In particular 
$$
Y_\eta(x)\le \frac{b\sqrt{\eta}}{|x|^\frac{3}{2}}.
$$
Hence, for \eqref{prooflemma3est2.58}, it suffices to show, for $\eta>0$ small enough and any $C>0$,
\begin{equation}\label{prooflemma3est2.59}
\E\bigg[\exp\bigg\{C b \sqrt\eta \int_0^1 \frac{\d s}{|W_s|^{3/2}}\bigg\}\bigg] <\infty.
\end{equation}
For this, we appeal to Khas'minski's lemma  which states that, if for a Markov process $\{\P^{\ssup x}\}$ and for a function $\widetilde V\ge 0$
$$
\sup_{x\in \R^3} \E^{\ssup x}\bigg\{\int_0^1 \widetilde V(W_s)\d s\bigg\}\leq \gamma<1
$$
then
$$
\sup_{x\in \R^3} \E^{\ssup x}\bigg\{\exp\bigg\{\int_0^1 \widetilde V(W_s)\d s\bigg\}\bigg\}\le \frac{\gamma}{1-\gamma} <\infty.
$$
Hence, to prove \eqref{prooflemma3est2.59}, we need to verify that 
$$
\begin{aligned}
\sup_{x\in \R^3} \E^{\ssup x}\bigg\{\int_0^1 \frac{\d \sigma}{|W_\sigma|^\frac{3}{2}}\bigg\}
=\sup_{x\in \R^3} \int_0^1 \d \sigma\int_{\R^3} \d y \,\, \frac{1}{|y|^\frac{3}{2}} \frac{1}{(2\pi \sigma)^\frac{3}{2} }\exp\bigg\{-\frac{(y-x)^2}{2\sigma}\bigg\}
<\infty.
\end{aligned}
$$
One can see that  
$$
\sup_{x\in \R^3} \int_{\R^3}\d y\frac{1}{|y|^\frac{3}{2}} \frac{1}{(2\pi \sigma)^\frac{3}{2} }\exp\bigg\{-\frac{(y-x)^2}{2\sigma}\bigg\}
$$ 
is attained at $x=0$ because we can rewrite the integral as 
$$ c\int_{\R^3} \exp\bigg\{-\frac{\sigma|\xi|^2}{2}+i\langle x,\xi\rangle\bigg\}\frac{1}{|\xi|^\frac{3}{2}}d\xi,
$$ 
where $c>0$ is a constant. When $x=0$, the integral reduces to $\int_0^1 \sigma^{-3/4} \,\,\d \sigma$, which is finite. This proves \eqref{check1} for $V(x)=1/|x|$ in $\R^3$. The proof for 
$V(x)=|x|^{-p}$ for $p< 2/(d-2)$ in $d\geq 3$ follows the same strategy and  is omitted to avoid repetition. 
For $V(x)=\delta_0(x)$ on $\R$ we have $\Lambda(x)=\int_0^1 \d s \, \delta_{W_s}(x)$, the Brownian local time at $x$. In this case the desired exponential moment \eqref{check1} follows from 
\cite[Chapter XII, 3.8]{RY91}. The proof of Lemma \ref{lemma:Lambda} is thus completed.\qed

\begin{appendix}
\section*{An entropy inequality}\label{appn} 
\begin{lemma}\label{lemma:Pinsker}
For any two probability measures $\mu,\nu$ on a measurable space $(X, \mathcal F)$, let
$$
\mathrm{Ent}(\mu|\nu)= \E^{\mu}\bigg[\log\bigg(\frac{\d\mu}{\d\nu}\bigg)\bigg]= \E^{\nu}\bigg[\frac{\d\mu}{\d\nu}\log\bigg(\frac{\d\mu}{\d\nu}\bigg)\bigg]$$
 if $\mu\ll \nu$ and $\mathrm{Ent}(\mu|\nu)=\infty$ otherwise. Then for any $\lambda>0$, 
\begin{equation}\label{eq2:Pinsker}
\E^\mu(f) \leq \lambda^{-1} \mathrm{Ent}(\mu|\nu)+ \lambda^{-1} \log\E^\nu \big(\e^{\lambda f} \big)
\end{equation} 
Moreover, 
\begin{equation}\label{eq3:Pinsker}
\|\mu-\nu\|_{\mathrm{TV}} \leq \sqrt{\frac 12 \mathrm{Ent}(\mu|\nu)}.
\end{equation}
\end{lemma} 
\begin{proof}
Although both estimates are well-known, we include it here for the sake of completeness. To prove the first estimate, note that by Jensen's inequality, 
$$
\lambda\E^\nu[f]- \mathrm{Ent}(\mu|\nu)= \E^\mu\bigg[\log\bigg(\e^{\lambda f}\, \frac{\d\nu}{\d\mu}\bigg)\bigg] \leq \log\big(\E^\nu[\e^{\lambda f}]\big)
$$
which proves \eqref{eq2:Pinsker}. To prove \eqref{eq3:Pinsker}, it suffices to show that for any probability density $\phi$ on $(X,\mathcal F,\mu)$ (i.e., $\phi\geq 0$ and $\int \phi=1$),
\begin{equation}\label{eq4:Pinsker}
\sup_{\|f\|_\infty\leq 1} \big|\int f \phi \d\mu- \int f \d\mu\big| \leq \sqrt{\frac 12 \int \phi\log \phi\d\mu}. 
\end{equation}
For any $f$ with $\|f\|_\infty\leq 1$ let us define, for any $\lambda\geq 0$, 
$$
m_f(\lambda):= \log \E^\mu\big[\e^{\lambda f}\big], \qquad \mbox{and }\,\,\,\d\mu_{\lambda,f}:= \frac{\e^{\lambda f}}{\E^\mu[\e^{\lambda f}]} \, \d\mu.
$$
Then, 
$$
\frac{\d}{\d\lambda} m_f(\lambda)= \E^{\mu_{\lambda,f}}[f] , \qquad\mbox{and}\,\,\,\, \frac{\d^2}{\d^2\lambda} m_f(\lambda)= \mathrm{Var}^{\mu_{\lambda,f}}\big[f\big] \leq \E^{\mu_{\lambda,f}}\big[f^2\big] \leq 1.
$$
By Taylor's theorem, $m_f(\lambda) \leq m_f(0)+ \lambda m_f^\prime(0)+ \lambda^2/2$, so that 
$$
\lambda^{-1} \log\E^\mu[\e^{\lambda f}] \leq \frac \lambda 2+ \E^\mu[f]
$$
The last estimate, together with the relative entropy inequality \eqref{eq2:Pinsker} implies that 
$$
\int(\phi-1) f \d\mu \leq \frac \lambda 2 + \frac 1 \lambda \int \phi\log\phi \d\mu. 
$$
Now optimizing over $\lambda$ implies 
$
\int(\phi-1) f \d\mu
\leq \sqrt{1/2 \int\phi\log\d\mu}
$
proving the desired estimate \eqref{eq4:Pinsker}.

\end{proof}

\end{appendix}




\begin{acks}[Acknowledgments]
It is a pleasure to thank Volker Betz and Herbert Spohn for their encouragement to pursue this work and many valuable discussions on the Nelson model. The author would also like to thank Erwin Bolthausen, Sabine Jansen and S.R.S. Varadhan for useful comments on an earlier version of the manuscript and Ofer Zeitouni for helpful discussions. Finally, the author would like to thank an anonymous referee for a very careful reading of the earlier version and pointing out a number of inaccuracies which led to a more elaborate version of our manuscript.
\end{acks}

\begin{funding}
The present work is supported by the Deutsche Forschungsgemeinschaft (DFG) under Germany's Excellence Strategy {\it EXC 2044--390685587, Mathematics M\"unster: Dynamics--Geometry--Structure}. 
\end{funding}

\end{document}